\documentclass[12pt]{amsart}
\usepackage{color}

\usepackage{amsmath}
\usepackage[dvips]{graphics}
\usepackage{amsfonts,amssymb}
\usepackage{enumitem,hyperref}
\usepackage[all]{xypic}

\newcounter{commentcounter}

\theoremstyle{plain}
\newtheorem*{theorem*}{Theorem}
\newtheorem*{lemma*} {Lemma}
\newtheorem*{corollary*} {Corollary}
\newtheorem*{proposition*} {Proposition}
\newtheorem{theorem}{Theorem}[section]
\newtheorem{lemma}[theorem]{Lemma}
\newtheorem{corollary}[theorem]{Corollary}
\newtheorem{proposition}[theorem]{Proposition}

\theoremstyle{remark}
\newtheorem*{remark}{Remark}

\newtheorem*{claim}{Claim}

\newcommand{\smfrac}[2]{\mbox{\footnotesize$\displaystyle\frac{#1}{#2}$}} % small medium frac
\newcommand{\tmfrac}[2]{\mbox{\large$\frac{#1}{#2}$}} % tiny medium frac

\theoremstyle{definition}

\def\tautwo{\tau^{(2)}}
\def\nng{\NN(G)}
\def\nnh{\NN(H)}

\def \R {\Bbb{R}}
\def \Z {\Bbb{Z}}
\def \C {\Bbb{C}}

\def\op{\operatorname}

\def\det{\op{det}}
\def\detr{\op{det}^{\op{r}}}

\def \NN{\mathcal{N}} \def \PP{\mathcal{P}}

\def \GG{\mathcal{G}}
\def \KK{\mathcal{K}}

\def\rz{\R[z^{\pm 1}]}

\def\dt{\,dt}

\def\wti{\widetilde}
\def\what{\widehat}

\def\eps{\epsilon}

\def\s{\sigma}

\def\Q{\Bbb{Q}}
\def\id{\op{id}}

\def\Z{\Bbb{Z}}
\def\C{\Bbb{C}}

\def\N{\Bbb{N}}
\def\l{\lambda}

\def\part{\partial}
\def\ll{\langle}
\def\rr{\rangle}

\def\g{\gamma}
\def\odd{\op{odd}}
\def\ev{\op{ev}}

\def\bp{\begin{pmatrix}}

\def\sm{\setminus}

\def\ep{\end{pmatrix}}
\def\bn{\begin{enumerate}}

\def\rk{\op{rank}}
\def\en{\end{enumerate}}
\def\ba{\begin{array}}
\def\ea{\end{array}}

\def\co{\colon}

\def\S{\Sigma}
\def\s{\sigma}

\def\fr12{\frac{1}{2}}

\def\im{\op{Im}}

\def\ker{\op{Ker}}
\def\be{\begin{equation}}
\def\ee{\end{equation}}
\def\tr{\op{tr}}
\def\t{\theta}

\def\hom{\op{Hom}}

\def\ng{{\mathcal{N}(G)}}

\def\ds{ds}
\def\ol{\overline}

\def\vol{\op{Vol}}
\def\genus{\op{genus}}

\def\gl{\op{GL}}
\def\zz{\Z[z^{\pm 1}]}
\def\rz{\R[z^{\pm 1}]}

\def\cmtbf#1{} \def\cmt#1{}

\begin{document}

\title{The $L^2$-Alexander torsion of 3-manifolds}

\author{J\'er\^ome Dubois}
\address{Universit\'e Blaise Pascal - Laboratoire de Math\'ematiques UMR 6620 - CNRS\\
Campus des C\'ezeaux - B.P. 80026\\
63171 Aubi\`ere cedex\\
France}
  \email{jerome.dubois@math.univ-bpclermont.fr}
  
\author{Stefan Friedl}
\address{Fakult\"at f\"ur Mathematik\\ Universit\"at Regensburg\\93040 Regensburg\\   Germany}
\email{sfriedl@gmail.com}

\author{Wolfgang L\"uck}
\address{Mathematisches Institut\\ Universit\"at Bonn\\
Endenicher Allee 60\\ 53115 Bonn\\ Germany}
\email{wolfgang.lueck@him.uni-bonn.de}
\date{\today}

\begin{abstract}
  We introduce $L^2$-Alexander torsions for 3-manifolds, which can be viewed as a
  generalization of the $L^2$-Alexander invariant of Li--Zhang. We state the
  $L^2$-Alexander torsions for graph manifolds and we partially compute them for fibered
  manifolds. We furthermore show that given any irreducible 3-manifold there exists a
  coefficient system such that the corresponding $L^2$-torsion detects the Thurston norm.
\end{abstract}

\maketitle

%==============================================================================
\section{Introduction}

Given a prime 3--dimensional manifold $N$ and $\phi \in H^1(N;\R)$, we use $L^2$-torsions to define an invariant 
$$\tautwo(N,\phi):\R^+\to [0,\infty)$$
which is called the \emph{full $L^2$-Alexander torsion} of $(N, \phi)$. We will see that $\tautwo(N,\phi)
(t=1)$ determines the volume of $N$. Here the volume of $N$ is the sum of the hyperbolic pieces in the JSJ--decomposition of $N$. In the paper, we are mostly interested in the limits of  $\tautwo(N,\phi)$ when $t$ goes to $0$ and to $\infty$. Especially, we will prove that for graph manifolds and fibered spaces these limits determine the Thurston norm of the manifold. As a corollary, we reprove a result obtained by Ben Aribi~\cite{BA13a, BA13b} which asserts that the $L^2$-Alexander torsion detects the unknot.

%==============================================================================
\subsection{The $L^2$-Alexander torsion}
An \emph{admissible triple} $(N,\phi,\g)$ consists of a prime orientable compact
3--dimensional manifold $N$ with empty or toroidal boundary, a class $\phi \in
H^1(N;\R)=\hom(\pi_1(N),\R)$ and a homomorphism $\g\co \pi_1(N)\to G$ such that $\phi\co
\pi_1(N)\to \R$ factors through $\g$. We say that an admissible triple $(N,\phi,\g)$ is
\emph{rational} if $\phi \in H^1(N;\Q)$ is a rational cohomology class.

Given an admissible triple $(N,\phi,\g)$ we use the $L^2$--torsion, see e.g.~\cite{Lu02}
for details, to introduce in Section~\ref{section:defl2alextorsion} the
\emph{$L^2$-Alexander torsion $\tautwo(N,\phi,\g)$} which is a function
\[
\tautwo(N,\phi,\g):\R^+\to [0,\infty).
\] 
We say that two functions $f,g\co \R^+\to
[0,\infty)$ are \emph{equivalent}, written as $f\doteq g$, if there exists an $r\in \R$,
such that $f(t)= t^r g(t)$ for all $t\in \R^+$.  The equivalence class of
$\tautwo(N,\phi,\g)$ is a well--defined invariant of $(N,\phi,\g)$.  If $\g$ is the
identity homomorphism, then we will drop it from the notation, i.e.,  we just write
$\tautwo(N,\phi)$.

As we explained in more detail in~\cite{DFL15}, the \emph{$L^2$-Alexander torsion
  $\tautwo(N,\phi,\g)$} can be viewed as a `twisted' invariant of the pair $(N,\phi)$, and
in particular, as we explain in ~\label{section:l2ordinaryalex}can be viewed as a generalization of the classical Alexander polynomial of a knot  cousin to the twisted Alexander polynomial~\cite{Li01,FV10} 
and the higher-order Alexander polynomials~\cite{Co04,Ha05} of
3-manifolds.

Given any $\phi\in H^1(N;\R)$, if we take $t=1$ and $\g=\id$ we obtain the usual
$L^2$-torsion of a 3-manifold. The following theorem is now a slight reformulation of a theorem by  the third author and Schick~\cite[Theorem~0.7]{LS99}.

\begin{theorem}\label{thm:ls99}
  If $N$ is a prime 3--manifold with empty or toroidal boundary, then for any $\phi\in
  H^1(N;\R)$ we have \be\label{equ:tauvol}
  \tautwo(N,\phi,\id)(t=1)\,\,=\,\,\exp\left(\smfrac{1}{6\pi} \vol(N)\right),\ee
where  $\vol(N)$ denotes the sum of the volumes of the hyperbolic pieces in the JSJ
decomposition. 
\end{theorem}

We make two remarks on the differences between the above formulation and the formulation
of~\cite[Theorem~0.7]{LS99}.  \bn
\item If $N$ is a prime 3--manifold $N$ with empty or toroidal boundary, then either
  $N\cong S^1\times D^2$ or the boundary of $N$ is incompressible.  (See
  e.g.,~\cite[p.~221]{Ne99}).
\item In this paper we also use a slightly different convention for $L^2$-torsions
  compared to~\cite{LS99}. Tracing through the differences one notices, that the
  $L^2$-torsions differ by a sign, a factor of $\frac{1}{2}$ and by taking the
  logarithm.
\en

% ==============================================================================
\subsection{The degree of the $L^2$-Alexander torsion}
 We are interested in the behavior of the $L^2$-Alexander torsion for the limits $t\to
0$ and $t\to \infty$.  We say that a function $f$ is \emph{monomial in the limit}
if there exist $d,D\in \R$ and non-zero real numbers $c,C$ such that
\[
\lim_{t\to 0} \frac{f(t)}{t^{d}}=c \mbox{ and } \lim_{t\to \infty}
\frac{f(t)}{t^D}=C.
\] 
We refer to $\deg f(t):=D-d$ as the \emph{degree} of $f$.
Furthermore we say $f$ is \emph{monic} if $c=C=1$.  

Note that the notion  of being monomial in the limit, being monic and the degree only depend on the equivalence class of the function.

% ==============================================================================
\subsection{Calculations of the $L^2$-Alexander torsion}
In order to state our results on $L^2$-Alexander torsions for certain classes of
3-manifolds we need one more definition.  Let $N$ be a 3--manifold and let $\phi\in
H^1(N;\mathbb{Z}) = \mbox{Hom}(\pi_1(N),\mathbb{Z})$. The \emph{Thurston norm} of $\phi$
is defined as
\[
x_N(\phi)=\min \{ \chi_-(\S)\, | \, \S \subset N \mbox{ properly embedded surface dual to }\phi\}.
\]
Here, given a surface $\S$ with connected components $\S_1\cup\dots \cup \S_k$, we define
$\chi_-(\S)=\sum_{i=1}^k \max\{-\chi(\S_i),0\}$.  Thurston~\cite{Th86} showed that $x_N$
defines a (possibly degenerate) norm on $H^1(N;\mathbb{Z})$. It can be extended to a norm
on $H^1(N;\mathbb{R} )$ which we also denote by $x_N$.

In Section~\ref{section:graph} we will prove the following theorem.

\begin{theorem}\label{thm:graphintro}
  Let $N\ne S^1\times D^2,S^1\times S^2$ be a graph manifold. For any non-trivial
  $\phi\in H^1(N;\R)$ and any representative $\tau$ of $\tautwo(N,\phi)$ we have
  \[
   \tau(t) \doteq \left\{ \ba{ll} 1, &\mbox{ if }t\leq 1 \\ t^{x_N(\phi)}, &\mbox{ if }t\geq 1.\ea \right.
   \] 
 In particular $\tautwo(N,\phi)$ is monomial in the limit and it is monic of degree $x_N(\phi)$.
\end{theorem}

Let $N$ be a 3--manifold and let $\phi\in H^1(N;\mathbb{Q}) =
\mbox{Hom}(\pi_1(N),\mathbb{Q})$ be non--trivial. We say that \emph{$\phi$ is fibered} if
there exists a fibration $p\colon N\to S^1$ and an $r\in \Q$ such that the induced map
$p_*\colon \pi_1(N)\to \pi_1(S^1)=\mathbb{Z}$ coincides with $r\cdot \phi$. In  Section~\ref{section:fibered} we will recall the definition of the entropy $h(\phi)\geq 1$ of a fibered class. With this
definition we can now formulate the following theorem which we will prove, in a somewhat more generalized form,  in
Section~\ref{section:fibered}.

\begin{theorem}\label{thm:l2fiberedintro}
  Let $(N,\phi,\g)$ be a rational admissible triple with $N\ne S^1\times D^2,S^1\times
  S^2$ such that $\phi\in H^1(N;\Q)$ is fibered.  We denote by $h(\phi)$ the entropy of
  the fibered class $\phi$. There exists a 
  representative $\tau$ of $\tautwo(N,\phi,\g)$ such that
  \[
  \tau(t) = \left\{ \ba{ll} 1, &\mbox{ if }t<\tmfrac{1}{h(\phi)}, \\ t^{x_N(\phi)}, &\mbox{ if }
  t> h(\phi).\ea \right.
  \] 
  In particular $\tautwo(N,\phi,\g)$ is monomial in the   limit and it is monic of degree $x_N(\phi)$.
\end{theorem}

If $N$ is hyperbolic and $\phi$ is a primitive fibered class, then Kojima--McShane~\cite[Theorem~1]{KM14} showed, with slightly different notation, that 
\[ \ln(h(\phi))\cdot x_N(\phi)\geq 8\cdot \smfrac{1}{6\pi} \vol(N).\]
By Theorems~\ref{thm:ls99} and~\ref{thm:l2fiberedintro} this translates into 
\[ \lim_{t\searrow h(\phi)}\tautwo(N,\phi,\id)(t)\geq \big(\tautwo(N,\phi,\id)(1)\big)^8.\]
We summarize everything we know about $\tautwo(N,\phi,\id)$ of a fibered class in Figure~\ref{fig:fibered-3-manifolds}.
\begin{figure}[h]
\begin{center}
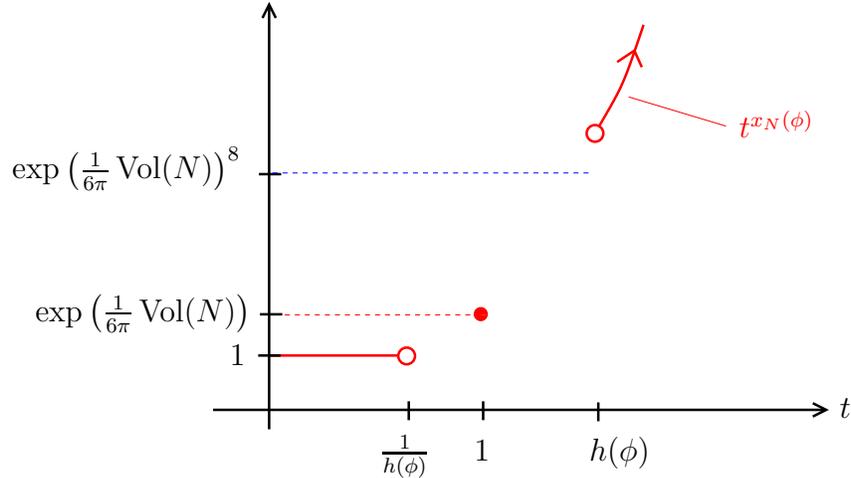
\caption{Partial graph of $\tautwo(N,\phi)$ for $N$ hyperbolic and $\phi$ fibered.} \label{fig:fibered-3-manifolds}
\end{center}
\end{figure}
At the moment we have no information for the values of $\tautwo(N,\phi,\id)$ for $t\in [1/h(\phi),1)$ and $t\in (1,h(\phi)]$.
We suspect that the function is continuous and convex.

%==============================================================================
\subsection{The symmetry of $L^2$-Alexander torsions}\label{section:symmetry}
For completeness we recall the main result of~\cite{DFL14}. In that paper we showed that
if $(N,\phi,\g)$ is an admissible triple, then for any representative $\tau$ of
$\tautwo(N,\phi,\g)$ there exists a $k\in \R$ such that 
\[ \label{equ:sym}
\tau(t^{-1})=t^{k}\cdot \tau(t).\] 
Put differently, the main theorem of \cite{DFL14} says that the $L^2$-Alexander
torsion is symmetric. This  result can in particular be viewed as an analogue of the fact
that the ordinary and the twisted Alexander polynomials of 3-manifolds are symmetric.
(See~\cite{HSW10,FKK12}).

%Now suppose that $\phi\in H^1(N;\Z)$.  Then there is an interesting relation to the Thurston norm $x_N(\phi)$.
%Namely, we can use the stronger equivalence relation 
%that two functions $f,g\co \R^+\to
%[0,\infty)$ are \emph{equivalent}, if there exists an $m\in \Z$,
%such that $f(t)= t^m g(t)$ for all $t\in \R^+$, and still $\tautwo(N,\phi,\g)$ is a well-defined invariant.
%In~\cite{DFL14} we prove that we have $k\in \Z$ in the equation~\eqref{equ:sym},
%and the reduction of $k$ mod $2$ is an invariant of $(N,\phi,\g)$,
%which turns out to be $x_N(\phi) \mbox{ mod }2$.

%==============================================================================
\subsection{The $L^2$-Alexander torsion of knot complements}

An important special case is given by knot exteriors.  Let $K$ be an oriented knot in
$S^3$. We denote by $\nu K$ an open tubular neighborhood of $K$ and we refer to
$X(K):=S^3\sm \nu K$ as the \emph{exterior of $K$}. Furthermore we denote by $\phi_K\in
H^1(X(K);\Z)=\hom(\pi_1(X(K)),\Z)$ the epimorphism which sends the oriented meridian to
$1$.  If $\g\co \pi_1(X(K))\to G$ is a homomorphism such that $(X(K),\phi_K,\g)$ forms an
admissible triple, then we write
\[
\tautwo(K,\g):=\tautwo(X(K),\phi_K,\g)\colon \R^+\to [0,\infty).
\] 
It follows from the symmetry result of Section~\ref{section:symmetry}
 that $\tautwo(K,\g)$ does not depend on the orientation of $K$.  Of
particular interest is the invariant $\tautwo(K):=\tautwo(K,\id)$. We will see in 
Section~\ref{section:li-zhang} that the resulting $L^2$-Alexander torsion is basically the same as
the $L^2$-Alexander invariant introduced by Li and Zhang~\cite{LZ06a,LZ06b}.

The calculations from the previous section also specialize to the case of knots.  For
example, Theorem~\ref{thm:graphintro} implies that for any (iterated) torus knot $K$ we
have
\[
\tautwo(K) \doteq ( t \mapsto \max\{1,t^{2\,\genus(K)-1}\}),
\] 
where $\genus(K)$ denotes the minimal genus
of a Seifert surface of $K$.  This equality was first proved by Ben Aribi~\cite{BA13a,BA13b} and 
generalizes an earlier result of the first author and Wegner~\cite{DW10,DW15}.  
The combination of Theorem~\ref{thm:graphintro} together with the
aforementioned work of the third author and Schick~\cite{LS99} gives us the following
theorem which states that the $L^2$-Alexander torsion detects the unknot. (We refer to Section~\ref{section:applications-to-knot-theory} for details.) This result was
first proved by Ben Aribi~\cite{BA13a,BA13b}.

\begin{theorem}\label{thm:detectsunknot} 
A knot $K\subset S^3$ is trivial if and only if $\tautwo(K)\doteq (t\mapsto \max\{1,t\}^{-1})$.
\end{theorem}

For knots it is enlightening to consider the coefficient system given by the abelianization 
$\g=\phi_K\in H^1(X(K);\Z)=\hom(\pi_1(X(K),\Z)$. In order to state the result we factor the Alexander
polynomial $\Delta_K(z)\in \zz$ as
\[
\Delta_K(z)=C\cdot z^m\cdot \prod_{i=1}^k (z-a_i),
\]
with some $C\in \Z\sm
\{0\}, m\in \Z$ and $a_1,\dots,a_k\in\C\sm \{0\}$.  In Section~\ref{section:l2ordinaryalex} we 
prove that
\[
\tautwo(K,\g)(t)\doteq C\cdot \prod_{i=1}^k \max\{|a_i|,t\}\cdot \max\{t,1\}^{-1}.
\] 
In
particular $\tautwo(K,\g)$ is a piecewise monomial function that is determined by the
ordinary Alexander polynomial.

%==============================================================================
\subsection{The $L^2$-Alexander torsion and the Thurston norm}
In this final section we want to relate the $L^2$-Alexander torsion to the Thurston norm
for more general types of 3-manifolds. In this context we can not show that
$L^2$-Alexander torsions are monomial in the limit.  In 
Section~\ref{section:functiondegree} we will therefore generalize the notion of degree from
functions that are monomial in the limit to more general types of functions.

With that definition of a degree we can show that $L^2$-Alexander torsions corresponding
to certain epimorphisms $\g$ give lower bounds on the Thurston norm.
More precisely, we will prove the following theorem in Section~\ref{section:thurston-norm-i}.

\begin{theorem} \label{thm:lowerboundintro} 
\label{thm:lowerbound}Let $(N,\phi,\g\co \pi_1(N)\to G)$ be an
  admissible triple with $N\ne S^1\times D^2$ and $N\ne  S^1\times S^2$. If $G$ is virtually abelian, i.e.,  if $G$ admits a finite index
  subgroup that is abelian, then
  \[
\deg \tautwo(N,\phi,\g) \leq x_N(\phi).
\]
\end{theorem}

Using the virtual fibering theorem of Agol~\cite{Ag08,Ag13}, Liu~\cite{Liu13},
Przytycki--Wise~\cite{PW12,PW14} and Wise~\cite{Wi12a,Wi12b} we will prove that there
exists a homomorphism $\g$ onto a virtually abelian group such that the $L^2$-Alexander
torsion in fact determines the Thurston norm.  More precisely, we have the following
theorem which is proved in Section~\ref{section:thurston-norm-ii}.

\begin{theorem}\label{thm:detectsnormintro}
\label{thm:detectsnorm}
  Let $N$ be a prime 3-manifold with empty or toroidal boundary that is not a closed graph
  manifold.  There exists an epimorphism $\g\co \pi_1(N)\to G$ onto a virtually
  abelian group such that the projection map $\pi_1(N)\to H_1(N;\Z)/\mbox{torsion}$
  factors through $\g$ and such that for any $\phi\in H^1(N;\R)$ the function
  $\tautwo(N,\phi,\g)$ is monomial in the limit with
  \[
  \deg \tautwo(N,\phi,\g)=x_N(\phi).
  \]
\end{theorem}

The paper is organized as follows. 
In Section~\ref{section:hilbert-modules} we introduce  the Fuglede--Kadison determinant and recall several of its key properties.
In Section~\ref{section:l2torsion} we use the Fuglede--Kadison determinant to introduce the $L^2$-torsion of a complex over a real group ring. 
The main object of study of this paper, the $L^2$-Alexander torsion, is introduced in Section~\ref{section:l2-alexander-torsion}.
In Section~\ref{section:basis-properties} we discuss several basic properties of the $L^2$-Alexander torsion, for example we show that it behaves well under the JSJ-decomposition. In Section~\ref{section:degree} we introduce the notion of the degree of a function $\R^+\to [0,\infty)$ which will later on play a key role. In Section~\ref{section:l2knots} we relate the $L^2$-Alexander torsion to the $L^2$-Alexander invariant of knots that was introduced by Li--Zhang~\cite{LZ06a,LZ06b,LZ08}. We also discuss in what sense the $L^2$-Alexander torsion can be viewed as a generalization of the classical Alexander polynomial of a knot. In Section~\ref{section:topological-info} we  calculate the $L^2$-Alexander torsion for graph manifolds, i.e.\ we prove
Theorem~\ref{thm:graphintro} which together with standard results in knot theory and Theorem~\ref{thm:ls99} gives a new proof of 
Theorem~\ref{thm:detectsunknot}. In Section~\ref{section:topological-info} we also give a partial calculation of the $L^2$-Alexander torsion for fibered manifolds, providing the proof of Theorem~\ref{thm:l2fiberedintro}.
In the last part of the paper we are interested in the relationship between the $L^2$-Alexander torsion and the Thurston norm.
More precisely, in Section~\ref{section:thurston-norm-i} we prove Theorem~\ref{thm:lowerbound} which states that degrees of the $L^2$-Alexander torsion give lower bounds on the Thurston norm
and in Section~\ref{section:thurston-norm-ii} we prove Theorem~\ref{thm:detectsnorm}, which says that $L^2$-Alexander torsions detect the Thurston norm.\\

\noindent \emph{Added in proof.} 
Very recently the authors~\cite{FL15} and independently Yi Liu~\cite{Liu15} showed that the full $L^2$-Alexander torsion detects the Thurston norm. Liu also proved several other interesting results, for example he showed that the full $L^2$-Alexander torsion is continuous and that it is monomial in the limit.

\subsection*{Conventions.}
We assume, unless we explicitly say otherwise, that all groups are finitely generated and that all 3--manifolds
are orientable, compact and connected and that the boundary is either empty or toroidal.

Given a ring $R$ we will view all modules as left $R$-modules, unless we say explicitly
otherwise.  Furthermore, given an $m \times n$-matrix $A$ over $R$, by a slight abuse of notation, we denote by $A \colon R^m \to R^n$ the $R$-homomorphism of left $R$-modules obtained by right multiplication with $A$ and thinking of elements in $R^m$ as the only row in a $1 \times m$-matrix.

\subsection*{Acknowledgments.}
The first author would like to warmly thanks University Paris Diderot--Paris 7 for its hospitality during the redaction of the paper, and also University Blaise Pascal for its financial support.
The second author gratefully acknowledges the support provided by the SFB 1085 `Higher
Invariants' at the University of Regensburg, funded by the Deutsche
  Forschungsgemeinschaft (DFG).   The
paper is financially supported by the Leibniz-Preis of the third author granted by the
 {DFG}. 
We wish to thank Fathi Ben Aribi for several helpful
conversations, we are especially grateful for pointing out an error in
the proof of Theorem~\ref{thm:l2fibered} in an earlier version of the paper. We also thank Greg McShane for pointing out the relationship of our work with \cite{KM14} and we thank Gerrit Herrmann for useful feedback. Finally we are also very grateful to the referee for reading an earlier version very carefully and for giving lots of helpful feedback that greatly improved the exposition of the paper.

%\newpage
%\tableofcontents
%\newpage

%==============================================================================
\section{Hilbert $\nng$-modules and the Fuglede--Kadison determinant}\label{section:hilbert-modules}
In this section we will recall the definition and some basic properties of 
Hilbert $\nng$-modules and the Fuglede--Kadison determinant. These will play a key role in the definition of the $L^2$-torsion of a chain complex in the next section.

At a first reading of the paper it is enough to know that given any group $G$ and any
matrix $A$ over $\R[G]$ (which is not necessarily a square matrix) one can, under slight
technical assumptions, associate to $A$ its Fuglede--Kadison determinant $\det_\nng(A)\in
\R^+$. Some of the key properties of the Fuglede--Kadison determinant are summarized in
Proposition~\ref{prop:detl2}.

%==============================================================================
\subsection{The  dimension of Hilbert $\nng$--modules}

Let $G$ be a group.
We denote by $\NN(G)$ the algebra of $G$--equivariant bounded linear operators from $l^2(G)$ to $l^2(G)$.
Following~\cite[Definition~1.5]{Lu02}
we define a \emph{Hilbert $\nng$-module}   to be a Hilbert
space $V$ together with a linear isometric left $G$-action such that there exists a
Hilbert space $H$ and an isometric linear $G$-embedding of $V$ into the tensor
product of Hilbert spaces $H \otimes l^2(G)$ with the $G$-action given by the $G$-action on the second factor.
A \emph{map of Hilbert $\nng$-modules $f\co  V \to W$} is a bounded $G$-equivariant operator.

For example, the Hilbert space $l^2(G)^m$ with the obvious left  $G$--action 
is a Hilbert $\nng$-module with $H=\R^m$. In the following we
will view elements of $l^2(G)^m$ as row vectors with entries in $l^2(G)$. In particular,
if $A$ is an $m\times n$-matrix over $\R[G]$, then $A$ acts by right multiplication on
$l^2(G)^m$. Here, as indicated already in the conventions, we view elements in $l^2(G)^m$ as row vectors. 
The matrix $A$  thus defines a map $ l^2(G)^m \to l^2(G)^n$. This map is in fact a map of
Hilbert $\nng$-modules.  \medskip

Let $V$ be a Hilbert $\nng$-module. One can associate to $V$ the \emph{von Neumann
  dimension $\dim_{\nng}(V)\in [0,\infty]$}.  We will not recall the definition, instead
we refer to~\cite[Definition~1.10]{Lu02} for details. We only note that the von Neumann
dimension has many of the usual properties of dimensions.  For example, if $V$ is a
Hilbert $\nng$-module, then $\dim_{\nng}(V)=0$ if and only if $V=0$. We refer 
to~\cite[Theorem~1.12]{Lu02} and~\cite[Theorem~6.29]{Lu02} for many more properties.

%==============================================================================
\subsection{Definition of the Fuglede--Kadison determinant}\label{section:def-fuglede-kadison}
\label{section:fuglede-kadison}
Let $G$ be a group and let $A$ be an $m\times n$-matrix over the group ring $\R[G]$. In
this section we recall the definition of the Fuglede--Kadison determinant of $A$.

As we mentioned above $A$ defines a map of Hilbert $\nng$-modules 
$l^2(G)^m\to l^2(G)^n$. 
We  consider the \emph{spectral density function of $A$} which is defined as 
\[
\ba{rcl} F_A: \R&\to &[0,\infty) \\
\l &\mapsto & {\small \sup\left\{ \dim_{\nng}(L)\,\left|\, \ba{l} 
\mbox{$L\subset l^2(G)^m$ a Hilbert $\nng$-submodule of $l^2(G)^m$}\\
\mbox{such that }\|Ax\|\leq \l\cdot \|x\|\mbox{ for all }x\in L\ea\right.\right\}}.\ea
\]
By~\cite[Section~2]{Lu02} the function $F_A$ is a monotone non--decreasing right--continuous
function.  Clearly $F_A(\l)=0$ for $\l<0$.

In the following let $F\colon \R\to [0,\infty)$ be a monotone non--decreasing,
right--continuous bounded function.  We then denote by $dF$ the unique measure on the
Borel $\s$--algebra on $\R$ which has the property that for a half open interval $(a,b]$
with $a<b$ we have
\[
dF((a,b]) =F(b)-F(a).
\]

Now we return to the $m\times n$-matrix $A$ over the group ring $\R[G]$.  
(We could consider more generally matrices over the von Neumann algebra $\nng$,
but we restrict ourselves to matrices over $\R[G]$.)

The \emph{Fuglede--Kadison determinant} of $A$ is  defined as
\[
\det_{\nng}(A):=\left\{ \ba{ll} \exp\left(\int_{(0,\infty)} \ln(\l)dF_A\right), &\mbox{ if }\int_{(0,\infty)} \ln(\l)dF_A>-\infty,\\
0,&\mbox{ if }\int_{(0,\infty)} \ln(\l)dF_A=-\infty.\ea \right.
\]
We say $A$ is of \emph{determinant class} if $\int_{(0,\infty)} \ln(\l)dF>-\infty$.  It
follows immediately from~\cite[Section~3.7]{Lu02} that this definition agrees with the
definition given, in a more general setup, in~\cite[Section~3.2]{Lu02}.

%==============================================================================
\subsection{Properties of the Fuglede--Kadison determinant}
\label{section:propfk}
For future reference we recall in the following two propositions some of the main
properties of the Fuglede--Kadison determinant.  Both propositions follow easily from the
definitions and from~\cite[Theorem~3.14]{Lu02}.

\begin{proposition}\label{prop:detl2}
Let $G$ be a group and let $A$ be a matrix over $\R[G]$.  The following assertions hold.
\begin{enumerate}
\item[$(1)$]\label{prop:detl2:swap}
 Swapping two columns or two rows of $A$ does not change the  Fuglede--Kadison determinant.
\item[$(2)$] \label{prop:detl2:add} Adding a column of zeros or a row of zeros does not change
  the Fuglede--Kadison determinant.
\item[$(3)$] \label{prop:detl2:right_mult} Right multiplication of a column by $\pm g$ with $g\in
  G$ does not change the Fuglede--Kadison determinant.
\item[$(4)$] \label{prop:detl2:induction} If $G$ is a subgroup of a group $H$, then we can also
  view $A$ as a matrix over $\R[H]$ and
\[
\det_{\nnh}(A)=\det_{\nng}(A).
\]
\item[$(5)$] \label{prop:detl2:trivial_group}  If $A$ is a matrix over $\R$
such that  the usual determinant $\det(A)$ is non-zero, then $\det_{\nng}(A)=|\det(A)|$.
\item[$(6)$] \label{prop:detl2:involution}  We denote by $\ol{A}$ the matrix which is obtained by applying the involution of $G$, 
$g\mapsto g^{-1}$, to each entry of $A$. Then
\[
\det_{\nng}\left(\ol{A}^t\right)=\det_{\nng}(A).
\]
 \end{enumerate}
\end{proposition}

Let $\what{G}\subset G$ be a subgroup of index $d$ and let $f\co V\to W$ be a homomorphism
between two based free left-$\R[G]$ modules. We pick representatives for
$G/\what{G}$. Multiplying the basis elements with all the representatives turns $V$ and
$W$ into based free left-$\R[\what{G}]$-modules.  In particular, if $A$ is a $k\times
l$-matrix over $\R[G]$ then the above procedure turns $A$ into $dk\times dl$-matrix which
we denote by $\iota_G^{\what{G}}(A)$. The fact that there is some slight indeterminacy in
the definition of $\iota_G^{\what{G}}(A)$, which stems from our need to pick representatives for $G/\hat{G}$, will not play a role.

\begin{proposition}\label{prop:detl2b}
  Let $\what{G}\subset G$ be a subgroup of finite index and let $f\co V\to W$ be a
  homomorphism between two based free left-$\R[G]$ modules. Then
  \[
  \det_{\NN(\what{G})}(f)=\det_{\nng}(f)^{[G:\what{G}]}.
  \]
  In particular $f$ is of determinant class viewed as a map of
  Hilbert-$\NN(\what{G})$-modules if and only if it is of determinant class viewed as a
  map of Hilbert-$\NN({G})$-modules.  Equivalently, if $A$ is a matrix over $\R[G]$, then
  $ \det_{\NN(\what{G})}\left(\iota_G^{\what{G}}(A)\right)=\det_{\nng}(A)^{[G:\what{G}]}$.
\end{proposition}

%==============================================================================
\subsection{The rank of a square matrix}
Let $G$ be a group and let $A$ be a $k\times k$-matrix over $\R[G]$. 
We define the \emph{rank of $A$} as follows:
\[
\rk_G(A):=k-\dim_{\nng}\left(l^2(G)^k/\ol{l^2(G)^kA}\right).
\]
Note that by~\cite[Lemma~2.11(11)]{Lu02} we have  $\rk_G(\ol{A}^tA)=\rk_G(A)$.
We say that $A$ has \emph{full rank} if $\rk(A)=k$.

We have the following characterization of matrices of full rank
which is an immediate consequence of~\cite[Theorem~1.12~(1)~and~(2)]{Lu02}.

\begin{lemma}\label{lem:fullrank}
  Let $G$ be a group and let $A$ be a $k\times k$-matrix over $\R[G]$. Then $A$ has full
  rank if and only if the map $l^2(G)^k\to l^2(G)^k$ given by right multiplication by $A$
  is injective. 
\end{lemma}

%==============================================================================
\subsection{Properties of the regular Fuglede--Kadison determinant}

Given a square matrix $A$ over $\R[G]$ we  define the \emph{regular Fuglede--Kadison determinant} as 
\[
\detr_{\NN(G)}(A):=\left\{ \ba{ll} \det_{\nng}(A), &\mbox{ if } G\mbox{ has full rank},\\
  0,&\mbox{ otherwise.} \ea \right.
\] 
The following proposition collects several key
properties of the regular Fuglede--Kadison determinant.  The proposition is again a
straightforward consequence of the definitions and of~\cite[Theorems~1.12~and~3.14]{Lu02}.

\begin{proposition}\label{prop:detl2square}
Let $G$ be a group and let $A$ be a $k\times k$-matrix over $\R[G]$. The following assertions hold.
\begin{enumerate}
\item[$(1)$] \label{prop:detl2square:full_rank} We have $\detr_{\NN(G)}(A)\ne 0$ if and only if $A$ is of determinant class and it has full rank. 
\item[$(2)$] \label{prop:detl2square:swap}  If we swap two columns or two rows of $A$, 
then the regular Fuglede--Kadison determinant stays unchanged.
\item[$(3)$] \label{prop:detl2square:multi_scalar}  If we multiply a row or a column of $A$ by $\l\in \R\sm \{0\}$, 
then the regular Fuglede--Kadison determinant is multiplied by $|\l|$.
\item[$(4)$] \label{prop:detl2square:multi_group_elem} Right multiplication of a column or left multiplication of a row by some $g\in G$ 
does not change the regular Fuglede--Kadison determinant.
\item[$(5)$] \label{prop:detl2square:induction} If $G$ is a subgroup of a group $H$, then we can also view $A$ also as a matrix over $\R[H]$ and
\[
\detr_{\NN(H)}(A)=\detr_{\NN(G)}(A).
\]
\item[$(6)$] \label{prop:detl2square:trivial_group}  If $A$ is a matrix over $\R$, then $\detr_{\nng}(A)=|\det(A)|$, 
where $\det(A)\in \R$ denotes the usual determinant. In particular, $\detr_{\nng}(\id)=1$.
\item[$(7)$] \label{prop:detl2square:finite_index}  If $\what{G}$ is a finite index subgroup of $G$, then
\[
\detr_{\NN(\what{G})}\left(\iota_G^{\what{G}}(A)\right)=\detr_{\NN(G)}(A)^{[G:\what{G}]}.
\]
\item[$(8)$] \label{prop:detl2square:composition}  If $A$ is a square matrix over  $\R[G]$ of the same size as $A$, then
\[
\detr_{\NN(G)}(A\cdot B)=\detr_{\NN(G)}(A)\cdot \detr_{\NN(G)}(B).
\]
\item[$(9)$] \label{prop:detl2square:su_formula}  If  $B$ is an $l\times l$-matrix and  $C$ is an $l\times k$-matrix, then
\[
\detr_{\NN(G)}\bp A&0 \\ C&B \ep =\detr_{\NN(G)}(A)\cdot \detr_{\NN(G)}(B).
\]
\end{enumerate}
\end{proposition}

%==============================================================================
\subsection{The class $\GG$}\label{section:classg}
In order to state the next theorem we need the notion of a `sofic' group. This class of groups was introduced by Gromov~\cite{Gr99}.
We will not recall the, somewhat technical, definition, but we note that by~\cite{Gr99} and \cite[Theorem~1]{ES06} the following hold:
\bn
\item the class of sofic groups  contains the class of  residually amenable groups,
\item any subgroup and any finite index extension of a sofic group is again sofic.
\en
It  follows from (1) that the following classes of groups are sofic:
\bn
\item residually finite groups, 
\item 3-manifold groups, since they are residually finite, see~\cite{Hem87}, and
\item virtually solvable groups.
\en
Here recall that if $\PP$ is a property of groups, then a group is said to be 
\emph{virtually $\PP$} if the group admits a finite index normal subgroup
that satisfies $\PP$.

The following theorem was proved by Elek and Szab\'o~\cite{ES05}.  (See
also~\cite{Lu94,Sc01,Cl99} for special cases.)

\begin{theorem} \label{thm:sofic}
Let $G$ be a group that is sofic. The following assertions hold.
\bn
\item[$(1)$] Any square matrix over $\Q[G]$ is of determinant class.
\item[$(2)$] If $A$ is a square matrix over $\Z[G]$, then $\det_{\nng}(A)\geq 1$.
\item[$(3)$] If $A$ is an invertible  matrix over $\Z[G]$, then $\det_{\nng}(A)=1$.
\en
\end{theorem}

\begin{proof}
  Let $G$ be a group that is sofic.  By the main result of~\cite{ES05} any square matrix
  $A$ over $\Z[G]$ is of determinant class with $\det_{\nng}(A)\geq 1$.

  If $A$ admits  an inverse matrix $B$ over $\Z[G]$, then it follows from
  Proposition~\ref{prop:detl2square} that
  \[
\det_{\nng}(A)\cdot \det_{\nng}(B)=\det^r_{\nng}(A)\cdot
  \det^r_{\nng}(B)=\det^r_{\nng}(AB)=\det^r_{\nng}(\id)=1.
\] 
  By the above both   $\det_{\nng}(A)$ and $\det_{\nng}(B)$ are at least one, it follows that
  $\det_{\nng}(A)=1$.

  Finally, if $A$ is a square matrix over $\Q[G]$, then we can write $A=r\cdot B$ with
  $r\in \Q$ and $B$ a matrix over $\Z[G]$.  It follows immediately from the aforementioned
  result of~\cite{ES05} and from the definitions that $A$ is also of determinant class.
\end{proof}

Now we denote by $\GG$ the class of all sofic groups $G$.  To the best of our knowledge it
is not known whether there exist finitely presented groups that are not sofic.
Moreover, we do not know whether  any matrix $A$ over any  real group ring  is of determinant class.

%===================================================
\subsection{The Fuglede--Kadison determinant and the Mahler measure}\label{section:mahler}
In general it is very difficult to calculate the Fuglede--Kadison determinant of a matrix
over a group ring $\R[G]$.  In this section we recall and make use of the well-known fact (see~\cite{Lu02,Ra12}) that if $G$
is free abelian, then the Fuglede--Kadison determinant can be expressed in terms of a Mahler
measure.

First, let $p\in \R[z_1^{\pm 1},\dots,z_k^{\pm 1}]$ be a multivariable Laurent
polynomial. If $p=0$, then its Mahler measure is defined as $m(p)=0$.  Otherwise the
\emph{Mahler measure} of $p$ is defined as
\[
m(p):=\exp\left(\frac{1}{(2\pi )^k}\int_{0}^{2\pi}\dots
  \int_{0}^{2\pi}\ln\left|p\left(e^{it_1},\dots,e^{it_k}\right)\right|\,dt_1\dots
  dt_k\right).
\] 
Note that the Mahler measure is multiplicative, i.e.,  for any non-zero
multivariable Laurent polynomials $p,q$ we have $m(pq)=m(p) \cdot m(q)$.  If $p\in \R[z^{\pm   1}]$ 
is a one-variable polynomial then we can write $ p(z)=D\cdot z^n\cdot \prod_{i=1}^l
(z-b_i)$, 
where $D\in \R, n\in \Z$ and
$b_1,\dots,b_l\in\C$.  It follows from Jensen's formula (see e.g.\ \cite[p.~207]{Ah78}) that \be \label{equ:jensen}
m(p)=|D|\cdot \prod_{i=1}^l\max\{1,|b_i|\}.\ee

If $H$ is a free abelian group of rank $k$ and $p\in \R[H]$ is non-zero, then we pick an
isomorphism $f\co \xymatrix@1{\Z^k \ar[r]^-\cong &H}$ which induces an isomorphism $\R[\Z^k]= \R[z_1^{\pm
  1},\dots,z_k^{\pm 1}]\cong \R[H]$ and we define the \emph{Mahler measure} of $p$ as $
m(p):=m(f_*^{-1}(p))$.  Note that this is independent of the choice of $f$.  Also note
that if $H$ is the trivial group and $p\in \R\sm \{0\}$, then $m(p)=|p|$.

The following lemma relates the regular Fuglede-Kadison determinant for free abelian groups to the Mahler measure.

\begin{lemma}\label{lem:detl2mahler}
Let $H$ be a free abelian group and let $A$ be a square matrix over $\R[H]$. Then
 \[\detr_{\NN(H)}(A)=m(\det_{\R[H]}(A))\]
where $\det_{\R[H]}(A)\in \R[H]$ is the usual determinant of the matrix $A$.
\end{lemma}

\begin{proof}
  Let $H$ be a free abelian group and let $A$ be a $k\times k$-matrix over $\R[H]$.  It
  follows from Lemma~\ref{lem:fullrank} that $A$ has full rank if and only if
  multiplication by $A$ is an injective map on $\R[H]^k$. But the latter is of course
  equivalent to $\det(A)\in \R[H]$ being non-zero.

  Now we suppose that $\det(A)\ne 0$. By the above we have
  $\detr_{\NN(H)}(A)=\det_{\nnh}(A)$, and the desired equality
  $\det_{\NN(H)}(A)=m(\det(A))$ is proved in~\cite[Section~1.2]{Ra12}, building
  on~\cite[Exercise~3.8]{Lu02}. 
\end{proof}

Given a finite set $S$ we denote by $\R[S]$ the $\R$-vector space spanned freely by the
elements of $S$. Given $n\in \N$ we denote by $M(n,\R[S])$ the set of all $n\times
n$-matrices with entries in $\R[S]$. Note that $M(n,\R[S])$ is a finite
dimensional real vector space and we endow it with the usual topology.  Now we have the
following useful corollary to Lemma~\ref{lem:detl2mahler}.

\begin{corollary}\label{cor:detr-continuous}
Let $G$ be a group that is virtually abelian. 
Then for any finite subset $S$ of $G$ the function
\[
\detr_{\NN(G)}\co M(n,\R[S])\to [0,\infty)
\]
is continuous.
\end{corollary}

\begin{proof}
  Since $G$ is virtually abelian (and finitely generated by our convention) there exists in particular a finite index subgroup $\what{G}$ that is torsion-free abelian.

  We pick representatives $g_1,\dots,g_d$ for $G/\what{G}$. Given a matrix $A$ over
  $\R[G]$ we define the matrix $\iota^{\what{G}}_G(A)$ over $\R[\what{G}]$ using this
  ordered set of representatives. It is straightforward to verify that there exists a
  finite subset $\what{S}$ of $\what{G}$ such that the map
\[
\iota^{\what{G}}_G\co M(n,\R[G])\to  M(dn,\R[\what{G}])
\]
restricts to a map
\[
\iota^{\what{G}}_G\co M(n,\R[S])\to  M(dn,\R[\what{S}])
\]
and that this map is continuous.  By  Proposition~\ref{prop:detl2square} it thus suffices to show that  
\[
\detr_{\NN(\what{G})}\co M(n,\R[\what{S}])\to [0,\infty)
\]
is continuous. But the continuity of this function is a consequence of 
Lemma~\ref{lem:detl2mahler} and  the continuity of the Mahler measure of multivariable  
polynomials of bounded degree, see~\cite[p.~127]{Bo98}.
\end{proof}

Finally we conclude with the following lemma.

\begin{lemma}\label{lem:det1minusg}
Let $G$ be a group, $g\in G$ an element of infinite order and let $t\in \R^+$.
Then
\[
\detr_{\NN(G)}(1-tg)=\max\{1,t\}.
\]
\end{lemma}

\begin{proof}
  By Proposition~\ref{prop:detl2square}~\eqref{prop:detl2square:multi_scalar}  we have $\detr_{\NN(G)}(1-tg)=\detr_{\NN(\ll
    g\rr)}(1-tg)$. By Lemma~\ref{lem:detl2mahler} we know that $\detr_{\NN(\ll
    g\rr)}(1-tg)$ equals the Mahler measure of $1-tg$, viewed as a polynomial in $g$.  By
  (\ref{equ:jensen}) we have $m(1-tg)=m((-t)(g-t^{-1}))=|-t|\cdot
  \max\{1,t^{-1}\}=\max\{1,t\}$.
\end{proof}

%================================================
\section{The $L^2$--torsion of  complexes over group rings}\label{section:l2torsion}
In this section we recall the definition of the $L^2$--torsion of  a complex over a real group ring $\R[G]$. This definition will then be used in the next section to define the $L^2$-Alexander torsion. 
We also provide two computational tools to compute the $L^2$--torsion which will be used later.

%================================================
\subsection{Definition of the $L^2$--torsion of  complexes over group rings}
First we recall some definitions (see~\cite[Definitions~1.16~and~3.29]{Lu02}).  Let $G$ be a group and let
\[
\xymatrix@1{0\ar[r] &  C_n\ar[r]^-{\partial_n} & C_{n-1} \ar[r] &\dots \ar[r] & C_1\ar[r]^-{\partial_1} & C_0\ar[r] & 0}
\]
be a complex of length $n$ of finitely generated free based left $\R[G]$-modules.  Here by
`based' we mean that all the $C_i$'s are equipped with a basis as free left
$\R[G]$-modules.  Note that the basing turns each $l^2(G) \otimes_{\R[G]} C_{i}$ naturally
into an $\nng$-module and the resulting boundary maps $\id\otimes \partial_i\co l^2(G)
\otimes_{\R[G]} C_{i}\to l^2(G) \otimes_{\R[G]} C_{i-1}$ are maps of $\nng$-modules.  Given
$i\in \{0,\dots,n\}$ we write
\[
\ba{rcl} Z_i(C_*)&:=&\ker\big\{l^2(G) \otimes_{\R[G]} C_{i}\xrightarrow{\id\otimes \partial_i} l^2(G)\otimes_{\R[G]} C_{i-1} \big\},\\[2mm]
B_i(C_*)&:=&\im\big\{l^2(G) \otimes_{\R[G]} C_{i+1}\xrightarrow{\id\otimes \partial_{i+1}} l^2(G)\otimes_{\R[G]} C_{i}\big\},\\[2mm]
H_i(C_*)&:=&Z_i(C_*)/\ol{B_i(C_*)},\ea 
\]
where $\ol{B_i(C_*)}$ denotes the closure of $B_i(C_*)$ in the Hilbert space $l^2(G) \otimes_{\R[G]} C_i$.
Furthermore we denote by
\[
b_i^{(2)}(C_*):=\dim_{\nng}H_i(C_*)
\]
the $i$-th $L^2$-Betti number of $C_*$.
We say that the complex $C_*$ is \emph{weakly acyclic} if all its $L^2$-Betti numbers vanish.

If the complex is not weakly acyclic, or if at least one of  the boundary maps is not of determinant class, then we
define $\tau^{(2)}(C_*):=0$.  (Note that this convention differs from the one used
in~\cite{Lu02}.)  If the complex is weakly acyclic and if the boundary maps are of
determinant class then its $L^2$--torsion is defined as follows:
\[
\tautwo(C_*):=\prod_{i=1}^n \det_{\nng}(\partial_i)^{(-1)^i}\in (0,\infty).
\] 
Note that
we take the \emph{multiplicative inverse of the exponential} of the $L^2$--torsion defined
in the monograph~\cite{Lu02}. Our convention of using the multiplicative inverse follows
the convention for 3-manifolds established in~\cite{Tu86,Tu01,Tu02a} and as we will see later on, our present choice matches the conventions used in the earlier literature \cite{BA13a,BA13b,DW10,DW15,LZ06a,LZ06b,LZ08} on $L^2$-Alexander invariants of knots.

%================================================
\subsection{Calculating  $L^2$--torsions using square matrices}

The following two lemmas are analogues to~\cite[Theorem~2.2]{Tu01}.

\begin{lemma}\label{lem:torsion2complex}
Let $G$ be a group.
Let
\[
\xymatrix@1{0\ar[r] &  \R[G]^{k} \ar[r]^-{B} & \R[G]^{k+l}\ar[r]^-{A} & \R[G]^l\ar[r] & 0}
\]
be a complex. Let $L\subset \{1,\dots,k+l\}$ be a subset of size $l$.
We write
\[
\ba{rcl} A(L)&:=&\mbox{rows in $A$ corresponding to $L$},\\
B(L)&:=&\mbox{result of deleting the columns of $B$ corresponding to $L$}.\ea 
\]
If $\detr_{\NN(G)}(A(L))\ne 0$, then
\[
\tautwo(\mbox{based complex})=\detr_{\NN(G)}(B(L))\cdot \detr_{\NN(G)}(A(L))^{-1}.
\]
\end{lemma}

\begin{proof}
We obtain the following short exact sequence of $\R [G]$-chain complexes (written as columns)
where $i$ and $p$ are the canonical inclusions and projections corresponding to $L$.
\[
\xymatrix@R0.7cm@C1cm{
%&
%\vdots \ar[d]
%&
%\vdots \ar[d]
%&
%\vdots \ar[d]
%&
%\\
0 \ar[r] 
&
0 \ar[r] \ar[d]
&
0 \ar[r] \ar[d]
&
0 \ar[r] \ar[d]
&
0
\\
0 \ar[r]
&
0 \ar[r] \ar[d]
&
\R[G]^k \ar[r]^{\id} \ar[d]^B
&
\R[G]^k  \ar[r] \ar[d]^{B(L)}
&
0
\\
0 \ar[r]
&
\R[G]^l \ar[r]^{p}\ar[d]^{A(L)}
&
\R[G]^{k+l} \ar[r]^{i} \ar[d]^A
&
\R[G]^k  \ar[r] \ar[d]
&
0
\\
0 \ar[r]
&
\R[G]^l \ar[r]^{\id} \ar[d]
&
\R[G]^{l} \ar[r] \ar[d]
&
0  \ar[r] \ar[d]
&
0
\\
0 \ar[r] 
&
0 \ar[r] % \ar[d]
&
0 \ar[r] % \ar[d]
&
0 \ar[r] %\ar[d]
&
0
%\\
%&
%\vdots
%&
%\vdots
%&
%\vdots
%&
}
\]
If we apply $l^2(G) \otimes_{\R[G]} -$, we obtain a short exact sequence of Hilbert $\nng$-chain complexes.
Now the claim follows by a direct  application of the weakly exact long $l^2$-homology sequences and 
the sum formula for $L^2$-torsion to it, see~\cite[Theorem~1.21 on page~27 and Theorem~3.35 on page~142]{Lu02}.
\end{proof}

\begin{lemma}\label{lem:torsion3complex}
  Let $G$ be a group.  Let
  \[
\xymatrix@1{ 0\ar[r] &  \R[G]^{j} \ar[r]^-{C} & \R[G]^{k} \ar[r]^-{B} &
  \R[G]^{k+l-j}\ar[r]^-{A} & \R[G]^l\ar[r] & 0}
\] 
be a complex. Let $L\subset
  \{1,\dots,k+l-j\}$ be a subset of size $l$ and $J\subset \{1,\dots,k\}$ a subset of size
  $j$.  We write
  \[
\ba{rcl} A(J)&:=&\mbox{rows in $A$ corresponding to $J$},\\
  B(J,L)&:=&\mbox{result of deleting the columns of $B$ corresponding to $J$}\\
  &&\mbox{and deleting the rows corresponding to $L$}\\
  C(L)&:=&\mbox{columns of $C$ corresponding to $L$}.\ea 
\] 
If $\detr_{\NN(G)}(A(J))\ne 0$
  and $\detr_{\NN(G)}(C(L))\ne 0$, then
  \[
\tautwo(\mbox{based complex})=\detr_{\NN(G)}(B(J,L))\cdot
  \detr_{\NN(G)}(A(J))^{-1}\cdot \detr_{\NN(G)}(C(L))^{-1}.
\]
\end{lemma}

\begin{proof}
The proof of this lemma is very similar to the proof of Lemma~\ref{lem:torsion2complex}. First one puts the given complex of length three into a vertical short exact sequence of complexes where the complex on top has length two and the complex at the bottom has length one. Then one applies  multiplicativity of $L^2$-torsions. Finally one applies Lemma~\ref{lem:torsion2complex} to the complex of length two. 
The end result is the desired formula. We leave the details to the reader.
\end{proof}

%===================================================
\subsection{The $L^2$-torsion and the Mahler measures}\label{section:mahlersub}
Given a free abelian group $H$ we  denote by $\R(H)$ the quotient field of $\R[H]$. For
$f=pq^{-1}\in \R(H)$ we define $m(f):=m(p)m(q)^{-1}$.  Given a chain complex of based
$\R[H]$-modules $C_*$ we denote by $\tau(C_*)\in \R(H)$ the Reidemeister torsion of
$\R(H)\otimes_{\R[H]} C_*$ as defined in~\cite[Section~I]{Tu01}. Note that by definition
$\tau(C_*)=0$ if and only if $\R(H)\otimes_{\R[H]} C_*$ is not acyclic.

We can now formulate the following useful proposition.

\begin{proposition}\label{prop:l2abelian}
  Let $H$ be a free abelian group and let $C_*$ be a chain complex of based
  $\R[H]$-modules. Then
  \[
\tautwo(C_*)=m\big(\tau(\R(H)\otimes_{\R[H]} C_*)\big).
\]
\end{proposition}

\begin{proof}
Let $H$ be a free abelian group and let $(C_*,c_*)$ be a chain complex of based $\R[H]$-modules. 
  We observe from~\cite[Lemma~1.34 on page~35]{Lu02} that $C_*$ is $L^2$-acyclic if and
  only if $C_*^{(0)} := \R(H)\otimes_{\R[H]} C_*$ is acyclic. 
(To be precise,   \cite[Lemma~1.34 on page~35]{Lu02} works with the quotient field $\C(H)$ over $\C[H]$, but it is clear that $\C(H)\otimes_{\R[H]} C_*$ is acyclic if and
    only if $\R(H)\otimes_{\R[H]} C_*$ is acyclic.)
  Put differently, $\tautwo(C_*)=0$ if and
  only if $m(\tau(C_*^{(0)}))=0$. Hence we can assume without loss of
  generality that both torsions are non-zero. This implies that $C_*^{(0)}$
  is contractible as an $\R(H)$-chain complex and we can choose an $\R(H)$-chain
  contraction $\gamma$. Then by \cite[Theorem~2.6]{Tu01} we have
  \[
  \tau(C_*^{(0)}) = \det_{\R(H)}\big((c+ \gamma)_{\ev} \colon C_{\ev}^{(0)} \to C_{\odd}^{(0)}\big) \quad \in
  \R(H)^{\times}.
  \]
  Clearing denominators we can find  an element $x \in \R[H]$ and $\R[H]$-homomorphisms $\gamma_n' \colon C_n \to C_{n+1}$
  such that over the quotient field $\R(H)$ the composite of $l_x \circ \gamma_n$ is
  $\gamma_n'$, where $l_x$ is left multiplication with $x$. We get
  \[
  \tau(C_*^{(0)} ) = \det_{\R[H]}((l_x \circ c+ \gamma')_{\ev} \colon C_{\ev} \to C_{\odd}) \cdot
  \det_{\R[H]}(l_x \colon C_{\odd} \to C_{\odd})^{-1} \quad \in \R(H)^{\times}.
  \]
  On the other hand, we conclude from~\cite[Lemma~3.41 on page~146]{Lu02} applied to the weak chain contraction given
  by $(\gamma',l_x)$ that
  \[
  \tautwo(C_*) = \det_{\nnh}((l_x \circ c+ \gamma')_{\ev} \colon C_{\ev} \to C_{\odd})
  \;\cdot \;\det_{\nnh}(l_x \colon C_{\odd} \to C_{\odd})^{-1} \quad \in (0,\infty).
  \]
  Now the claim follows from Lemma~\ref{lem:detl2mahler}.
\end{proof}

%\begin{corollary}
%Let $G$ be a group which admits a finite index subgroup $H$ which is free abelian. 
%\end{corollary}
%
%\begin{proof}
%
%We pick representatives $g_1,\dots,g_d$ for $G/\what{G}$. Given a matrix $A$ over $\R[G]$ we now 
%define the matrix $\iota^{\what{G}}_G(A)$ over $\R[\what{G}]$ using this ordered set of representatives. 
%\end{proof}

%==============================================================================
\section{Admissible triples and the $L^2$-Alexander torsion}\label{section:l2-alexander-torsion}

After the preparations from the last two sections we can now introduce the  $L^2$-Alexander torsion of 3-manifolds. 
%==============================================================================
\subsection{Admissible triples}

Let $\pi$ be a group,  $\phi \in \hom(\pi,\R)$ a non-trivial homomorphism and
$\g\co \pi\to G$  a homomorphism.
We say that $(\pi,\phi,\g)$ form an \emph{admissible triple}  if $\phi\co \pi\to \R$ factors through $\g$, i.e.,  
if there exists a homomorphism
 $G\to \R$ such that the following diagram commutes:
 \[
\xymatrix{ \pi\ar[dr]_{\phi}\ar[r]^\g & G\ar[d] \\ &\R.}
\]

Note that if  $\g\co \pi\to G$ is a homomorphism such that the projection map
$\pi\to H_1(\pi;\Z)/\mbox{torsion}$ factors through $\g$,
then $(\pi,\phi,\g)$ is an admissible triple for any
$\phi\in \hom(\pi,\R)$.

% rational coefficients
If $N$ is a prime 3--manifold, $\phi \in
H^1(N;\R)=\hom(\pi_1(N),\R)$ and $\g\co \pi_1(N)\to G$, then we say that $(N,\phi,\g)$
form an \emph{admissible triple} if $(\pi_1(N),\phi,\g)$ form an admissible triple. Note
that this is consistent with the definition given in the introduction.

Let $(\pi,\phi,\g\co \pi\to G)$ be an admissible triple and let $t\in \R^+$.  We consider
the ring homomorphism
\[
\ba{rrcl} \kappa(\phi,\g,t)\co &\Z[\pi] &\to& \R[G] \\
&\sum\limits_{i=1}^n a_iw_i&\mapsto & \sum\limits_{i=1}^n a_it^{\phi(w_i)}\g(w_i).\ea 
\]
  Note that this ring homomorphism allows us to
view $\R[G]$ and $\nng$ as $\Z[\pi]$-right modules via right multiplication.  Given a
matrix $A=(a_{ij})_{ij}$ over $\Z[\pi]$ we furthermore write
\[
\kappa(\phi,\g,t)(A):=\big(\kappa(\phi,\g,t)(a_{ij})\big)_{ij}.
\]

%==============================================================================
\subsection{Definition of the $L^2$-Alexander torsion of CW--complexes and manifolds}
\label{section:defl2alextorsion}
Let $X$ be a finite CW--complex.  We write $\pi=\pi_1(X)$.  Let $\phi\in
H^1(X;\R)=\hom(\pi,\R)$ and let $\g\co \pi\to G$ be a homomorphism to a group such that $\phi$ factors through $\gamma$.  Finally
let $t\in \R^+$. Recall that $\kappa(\phi,\g,t)\colon \Z[\pi]\to \R[G]$ defines a right
$\Z[\pi]$-module structure on $\R[G]$ and $\nng$.  Now we denote by $\wti{X}$ the
universal cover of $X$.  The deck transformation induces a natural left $\Z[\pi]$--action
on $C_*(\wti{X})$. We consider the chain complex $\R[G]\otimes_{\Z[\pi]}C_*(\wti{X})$ of
left $\R[G]$-modules, where the $\R[G]$-action is given by left multiplication on $\R[G]$.  Now we pick an ordering and an orientation of the cells of $X$
and we pick a lift of the cells of $X$ to $\wti{X}$.  Note that the chosen lifts,
orderings and orientations of the cells endow each $\R[G]\otimes_{\Z[\pi]}C_i(\wti{X})$
with a basis as a free left $\R[G]$-module.  We then denote by
\[
\tau^{(2)}(X,\phi,\g,t)\in [0,\infty)
\] 
the corresponding torsion, as defined in
Section~\ref{section:l2torsion}.  Thus we obtain a function
\[
\ba{rccl} \tau^{(2)}(X,\phi,\g)\co & \R^+&\to & [0,\infty)\\
&t&\mapsto&\tau^{(2)}(X,\phi,\g,t)\ea 
\] 
that we call the \emph{$L^2$-Alexander torsion of $(X,\phi,\gamma)$}.  It follows
immediately from the definitions, Proposition~\ref{prop:detl2} that the function $\tautwo(X,\phi,\g)$ does not depend
on the orderings and the orientations of the cells. On the other hand,
using~\cite[Theorem~3.35 (5)]{Lu02} one can easily show that a change of lifts changes the
$L^2$-Alexander torsion function by multiplication by $t\mapsto t^r$ for some $r \in \R$.  Put
differently, the equivalence class of $\tautwo(X,\phi,\g)\co \R^+\to [0,\infty)$ is a
well-defined invariant of $(X,\phi,\g)$.

 Let $(N,\phi,\g)$ be an admissible triple.   We
 pick a CW--structure $X$ for $N$. We define
 \[
\tautwo(N,\phi,\g):=\tautwo(X,\phi,\g)\co \R^+\to [0,\infty).
\] 
A priori this definition depends on the choice of the CW--structure, but fortunately the following lemma says that its equivalence class is in fact an invariant of $(N,\phi,\g)$.

\begin{lemma}
The
 equivalence class of $ \tautwo(N,\phi,\g)$ is a well-defined invariant of $(N,\phi,\g)$.
\end{lemma}

\begin{proof}
The statement of the lemma follows from a standard circle of ideas. Therefore we only give a sketch of the proof. Suppose $X$ and $Y$ are two CW--structures for $N$. We denote by $f\colon X\to Y$ the corresponding homeomorphism. The argument of~\cite[Theorem~3.96~(1)]{Lu02} shows that there exists a square matrix $A$ over $\Z[\pi_1(N)]$ that represents the Whitehead torsion of $f$ in $\op{Wh}(\pi_1(N))$ and such that 
\[ \tautwo(X,\phi,\g)=\tautwo(Y,\phi,\g)\cdot \detr_{\ng}(\kappa(\phi,\g,t)(A))^{-1}.\]
By Chapman's Theorem~\cite[Theorem~1]{Ch74} the Whitehead torsion of $f$ is trivial. This implies that $A$ represents the trivial element in the Whitehead group $\op{Wh}(\pi_1(N))$. This in turn means by \cite[Lemma~1.1]{Mi66} that there exists an $n$ such that 
\[ \bp A& 0\\ 0&\id_n\ep\,\,=\,\, D\cdot \prod\limits_{i=1}^k E_i,\]
where $D$ is a diagonal matrix with diagonal entries in $\pi_1(N)$
and $E_1,\dots,E_k$ are  elementary matrices, i.e.\ they are square matrices that agree with the identity matrix except for one off-diagonal entry. It follows from  Proposition~\ref{prop:detl2square} (4), (8) and (9)
that 
\[\detr_{\ng}(\kappa(\phi,\g,t)(A))\doteq 1,\]
which by the above proves the desired equality of $L^2$-Alexander torsions.
\end{proof}

 If $\g=\id\co \pi_1(N)\to \pi_1(N)$ is the identity map, then we drop $\g$ from the
 notation, i.e.,  we write $\tautwo(N,\phi):=\tautwo(N,\phi,\id)$ and we refer to it as the
 \emph{full $L^2$-Alexander torsion of $(N,\phi)$}. 

 \begin{remark}
   In the above discussion we restricted ourselves to $t\in \R^+$. Verbatim the same
   discussion shows that we could also take $t\in \C\sm \{0\}$, which then gives rise to a
   function $ \tautwo(N,\phi,\g)\co \C\sm \{0\}\to [0,\infty)$.  But it follows from the
   argument of Li--Zhang~\cite[Theorem~7.1]{LZ06a} and Dubois--Wegner~\cite[Proposition~3.2]{DW15} that for any $t\in \C\sm \{0\}$ we have $
   \tautwo(N,\phi,\g)(t)= \tautwo(N,\phi,\g)(|t|)$.  Therefore we do not loose any
   information by restricting ourselves to viewing $\tautwo(N,\phi,\g)$ as a function on
   $\R^+$.
 \end{remark}

%==============================================================================
\section{Basic properties of the $L^2$-Alexander torsion}\label{section:basis-properties}
In this section we will state several results on $L^2$-Alexander torsions which can be proved easily using standard results on $L^2$-torsions. These results will nonetheless be crucial in our later discussions.

First, we recall that  we say that two functions $f,g\co \R^+\to [0,\infty)$ are \emph{equivalent},
written as \emph{$f\doteq g$},  if there exists an $r\in \R$, such that
$f(t)= t^r g(t) \mbox{ for all }t\in \R^+$.  Note that if two functions are equivalent, then the evaluations at $t=1$ agree.

The following lemma is an immediate consequence of the definitions and of
Proposition~\ref{prop:detl2}~(5).

\begin{lemma}\label{lem:injective}
Let $(N,\phi,\g\co \pi\to G)$ be an admissible triple and let  $\varphi\co G\to H$ be a monomorphism.
Then
\[
\tau^{(2)}(N,\phi,\varphi\circ \g)\doteq \tau^{(2)}(N,\phi,\g).
\]
\end{lemma}

The lemma in particular shows that for $L^2$-Alexander torsion we can restrict ourselves
to $\gamma$ being an epimorphism.  The next lemma follows immediately from the
definitions:

\begin{lemma}\label{lem:multiple}
Let $(N,\phi,\g)$ be an admissible triple, and let $r\in \R$, then
\[
\tautwo(N,r\phi,\g)(t) \doteq \tautwo(N,\phi,\g)(t^r).
\]
\end{lemma}

\begin{lemma}\label{lem:finitecover}
  Let $(N,\phi,\g\co \pi=\pi_1(N)\to G)$ be an admissible triple.  Let $p\colon
  \what{N}\to N$ be a finite regular cover such that $\ker(\g)\subset
  \what{\pi}:=\pi_1(\what{N})$.  We write $\what{\phi}:=p^*\phi$ and we denote by
  $\what{\g}$ the restriction of $\g$ to $\what{\pi}$.  Then
  \[
\tau^{(2)}(\what{N},\what{\phi},\what{\g})(t)\doteq
  \left(\tautwo(N,\phi,\g)(t)\right)^{[\what{N}:N]}.
\]
\end{lemma}

\begin{proof}
  We write $\pi=\pi_1(N)$ and $\what{\pi}:=\pi_1(\what{N})$.  We first note that by
  Lemma~\ref{lem:injective} we can and will assume that $\g$ is surjective.  Now we write
  $\what{G}:=\im(\what{\g})$.

  We pick a CW--structure $X$ for $N$ and we denote by $\what{X}$ the cover of $X$
  corresponding to the finite cover $\what{N}$ of $N$. We furthermore denote by $\wti{X}$
  the universal cover of $X$, which is of course also the universal cover of $\what{X}$.

  We pick lifts of the cells of $X$ to $\wti{X}$. These turn $\R[G]\otimes_{\Z[\pi]}
  C_*(\wti{X})$ into a chain complex of based free left $\R[G]$-modules. We also pick
  representatives for $\pi/\what{\pi}$.  By taking all the translates of the above lifts
  of the cells by all the representatives we can view $\R[\what{G}]
  \otimes_{\Z[\what{\pi}]} C_i(\wti{X})$ as a chain complex of based left
  $\R[\what{G}]$-modules. For the remainder of this proof we view $\tautwo(N,\phi,\g)$ and
  $\tau^{(2)}(\what{N},\what{\phi},\what{\g})$ as defined using these bases.

  Now we fix a $t\in \R^+$. Henceforth we view $\R[G]$ as a right $\Z[\pi]$-module via
  $\kappa(\phi,\gamma,t)$, and we view $\R[\what{G}]$ as a right $\Z[\what{\pi}]$-module via
  $\kappa(\what{\phi},\what{\gamma},t)$.  For each $i$ we  consider the map
  \[
   \ba{rcl}\R[\what{G}] \otimes_{\Z[\what{\pi}]}C_i(\wti{X})&\to &\R[{G}]\otimes_{\Z[\pi]}C_i(\wti{X})\\
\sum_i p_i\otimes \sigma_i&\mapsto & \sum_i p_i\otimes\sigma_i.\ea 
  \]
  It is  straightforward to see that these maps are well-defined maps of left
  $\R[\what{G}]$-modules.  Furthermore, it follows easily from the assumption
  $\ker(\g)\subset \what{\pi}:=\pi_1(\what{N})$ that these maps are in fact isomorphisms
  of left $\R[\what{G}]$-modules. Finally note that the maps are obviously chain maps.

  It  follows from~\cite[Theorem~1.35~(9)]{Lu02} that
  $\R[\what{G}]\otimes_{\Z[\what{\pi}]}C_i(\wti{X})$ is weakly acyclic if and only if
  $\R[G]\otimes_{\Z[\pi]}C_i(\wti{X})$ is weakly acyclic. Thus we can restrict ourselves
  to the case that both are weakly acyclic.  It then follows from
  Proposition~\ref{prop:detl2b} that
  \[
   \ba{rcl}
  \tau^{(2)}(\what{N},\what{\phi},\what{\g})(t)=\tau\left(\R[\what{G}]\otimes_{\Z[\what{\pi}]}C_i(\wti{X})\right)^{[\what{N}:N]}&=&
  \left(\tau\big(\R[G]\otimes_{\Z[\pi]}C_i(\wti{X})\big)\right)^{[\what{N}:N]}\\&=&\left(\tautwo(N,\phi,\g)(t)\right)^{[\what{N}:N]}.\ea
   \]
\end{proof}

It follows immediately from the definitions that if $(N,\phi,\g)$ is an admissible triple,
then so is $(N,-\phi,\g)$ and $\tautwo(N,-\phi,\g)(t)\doteq \tautwo(N,\phi,\g)(t^{-1})$.
In~\cite{DFL14} we proved the following theorem, which together with the above
discussion implies that $\tautwo(N,-\phi,\g)\doteq \tautwo(N,\phi,\g)$.

\begin{theorem}\label{thm:sym}
  Let $(N,\phi,\g)$ be an admissible triple and let $\tau$ be a representative of
  $\tautwo(N,\phi,\g)$. Then there exists an $r \in \R$ such that
  \[
\tau(t^{-1})=t^r \cdot \tau(t)\mbox{ for any }t\in \R_{>0}.
\] 
Furthermore, if $\phi  \in H^1(N;\Z)$, then there exists a representative $\tau$ of $\tautwo(N,\phi,\g)$ and an
  $n\in \Z$ with $n\equiv x_N(\phi) \mbox{ mod }2$ such that
  \[
\tau(t^{-1})=t^n\cdot \tau(t)\mbox{ for any }t\in \R_{>0}.
\]
\end{theorem}

We conclude this section with a discussion of the $L^2$-Alexander torsions of 3-manifolds
with a non-trivial JSJ decomposition.

\begin{theorem}\label{thm:jsj}
Let $N$ be a prime 3--manifold and $\phi\in H^1(N;\R)$.
We denote by $T_1,\dots,T_k$ the collection of JSJ tori and we denote by $N_1,\dots,N_l$ the JSJ pieces.
Let $\g\co \pi_1(N)\to G$ be a homomorphism such that the restriction to each JSJ torus has infinite image.
For $i=1,\dots,l$ we denote by $\phi_i\in H^1(N_i;\R)$ and $\g_i\colon \pi_1(N_i)\to G$ 
the restriction of $\phi$ and $\g$ to $N_i$.
Then
\[
\tautwo(N,\phi,\g) \doteq \prod_{i=1}^l \tautwo(N_i,\phi_i,\g_i).
\]
%In particular if $\tautwo(N_i,\phi,\g_i)=0$ for some $i$, then $\tautwo(N,\phi,\g)=0$.
\end{theorem}

In the proof of Theorem~\ref{thm:jsj} we will need the following lemma.

\begin{lemma}\label{lem:l2torus}
Let $T$ be a torus, let $\phi\in H^1(T;\R)$ and let $\g\co \pi_1(T)\to G$ be a homomorphism with infinite image such that $\phi$ factors through $\g$.
Then
\[
\tautwo(T,\phi,\g) \doteq 1.
\]
\end{lemma}
\begin{proof}
  We first note that by Lemma~\ref{lem:injective} we can assume that $\g$ is
  surjective. In particular this implies $G$ is an infinite, finitely generated abelian
  group. Note that $G$ contains a finite index subgroup which is free abelian.  Since a
  finite cover of a torus is once again a torus we can by Lemma~\ref{lem:finitecover}
  assume, without loss of generality, that $G$ is already free abelian.

  Now we pick a CW-structure for $T$ with one 0--cell $p$, two 1--cells $x,y$ and one
  2--cell $i$.  We write $\pi=\pi_1(T)$, we denote $\wti{T}$ the universal cover of $T$
  and we denote by $x$ and $y$ the elements in $\pi$ defined by the eponymous cells. Note
  that for appropriate lifts of the cells the based chain complex $C_*(\wti{T}; \Z)$ is
  isomorphic to
  \[0\to \Z[\pi]\xrightarrow{\bp 1-y&x-1\ep} \Z[\pi]^2\xrightarrow{\bp
    1-x\\1-y\ep}\Z[\pi]\to 0.\] 
     Let $t\in \R^+$.  Now
  we tensor the above chain complex with $\R[G]$, viewed as a $\Z[\pi]$-module via the
  representation $\kappa(\phi,\gamma,t)$. We  obtain the chain complex
  \[
C_*\,\,\,:= \,\,\, 0\to \R[G]\xrightarrow{\bp
    1-t^{\phi(y)}\g(y)&t^{\phi(x)}\g(x)-1\ep} \R[G]^2\xrightarrow{\bp
    1-t^{\phi(x)}\g(x)\\1-t^{\phi(y)}\g(y)\ep}\R[G]\to 0.
\] 
of based $\R[G]$-modules.
It follows immediately from Lemmas~\ref{lem:det1minusg} and~\ref{lem:torsion2complex}  that the $L^2$-torsion of this complex is always one.
Put differently, $ \tautwo(T,\phi,\g)\doteq
    1$.

    % Lemmas~\ref{lem:det1minusg} and~\ref{lem:torsion2complex} that the torsion of the
    % above chain complex is equal to one.
  \end{proof}

  Now we can give the proof of Theorem~\ref{thm:jsj}.

\begin{proof}[Proof of Theorem~\ref{thm:jsj}]
  We can and will pick a CW-structure for $N$ such that each JSJ torus, and thus also each
  JSJ component, corresponds to a subcomplex.

  Let $t\in \R^+$. We view $\nng$ as a right $\Z[\pi_1(N)]$-module via
  $\kappa(\phi,\gamma,t)$ and as a module over each $\Z[\pi_1(T_i)]$ and each
  $\Z[\pi_1(N_i)]$ via restriction.  Since the restriction of $\g$ to each JSJ torus has
  infinite image it follows from Lemma~\ref{lem:l2torus} that each
  $\nng\otimes_{\Z[\pi_1(T_i)]}C_*(\wti{T_i})$ is weakly acyclic.  It then follows from a
  Mayer--Vietoris argument that the chain complex $\nng\otimes_{\Z[\pi_1(N]}C_*(\wti{N})$
  is weakly acyclic if and only if all of the $\nng\otimes_{\Z[\pi_1(N_i)]}C_*(\wti{N_i})$
  are weakly acyclic.  In particular the theorem holds if one of the chain complexes is
  not weakly acyclic.

  Thus we can now assume that all of the above chain complexes are weakly acyclic.  The
  theorem then follows from Lemma~\ref{lem:l2torus} and the multiplicativity of
  $L^2$-torsions for short exact sequences, see~\cite[Theorem~3.35~(1)]{Lu02}. (Recall
  once again that the $L^2$-torsion in~\cite{Lu02} is minus the logarithm of our
  $L^2$-torsion.)
\end{proof}

%==============================================================================
\section{The  degree of  functions}
\label{section:degree}\label{section:functiondegree}
In this section we introduce the degree of a function $\R^+\to [0,\infty)$. Later on we will study the degree of the $L^2$-Alexander torsion and we will say that the degree of the function plays a role similar to the degree of a polynomial.
\smallskip

Let  $f\co \R^+\to [0,\infty)$ be a function.
If $f(t)=0$ for arbitrarily small $t$, then we define the \emph{degree of $f$ at $0$}  to be $\deg_0(f) :=\infty$.
 Otherwise we define
the \emph{degree of $f$ at $0$}  to be
\[
\deg_0(f) := \liminf_{t \to 0} \frac{\ln(f(t))}{\ln(t)}\in \R\cup \{-\infty\}. \]
Similarly, if $f(t)=0$ for arbitrarily large $t$, then we  define the \emph{degree of $f$ at $\infty$} to be $-\infty$. 
Otherwise  we  define the \emph{degree of $f$ at $\infty$} as 
\[
\deg_{\infty}(f) := \limsup_{t \to \infty} \frac{\ln(f(t))}{\ln(t)}\in \R\cup\{\infty\}.
\]
We follow the usual convention of extending addition on $\R$ partly to $\R\cup \{-\infty\}\cup \{\infty\}$, i.e.
\bn
\item for $a\in \R$ we define $a+\infty:=a$ and $a+(-\infty):=-\infty$, and 
\item we define $\infty+\infty:=\infty$ and $-\infty+(-\infty):=-\infty$.
\en
As usual we also define $a-b:=a+(-b)$. 
If $\deg_\infty(f)- \deg_0(f)$ is defined, then  we define the \emph{degree of $f$} as
\[
\deg(f):= \deg_\infty(f)- \deg_0(f).
\]
If $\deg_\infty(f)- \deg_0(f)$ is undefined, then  we set $\deg(f):=-\infty$.

In the following we say that a function $f\co \R^+\to [0,\infty)$ is \emph{piecewise monomial}
if we can find $0=t_0 <t_1<t_2\dots <t_k<t_{k+1}:=\infty$, 
 $d_0,\dots,d_k\in \Z$ and furthermore non-zero real numbers $C_0,\dots,C_k$ such that
\[
f(t)=C_it^{d_i}\mbox{ for all }t\in [t_i,t_{i+1})\cap \R^+.
\]
We say that a function $f\co \R^+\to [0,\infty)$ is \emph{eventually monomial}
if there exist $0=s<S<\infty$, $d,D\in \R$ and non-zero  real numbers $c,C$  such that
\[
f(t)=ct^{d}\mbox{ for  }t\in (0,s) \mbox{ and } f(t)=Ct^D \mbox{ for }t\in (S,\infty).
\]
Finally we recall that a function $f\co \R^+\to [0,\infty)$ is \emph{monomial in the limit}
if there exist $d,D\in \R$ and non-zero real numbers $c,C$  such that
\[
\lim_{t\to 0} \frac{f(t)}{t^{d}}=c \mbox{ and } \lim_{t\to \infty} \frac{f(t)}{t^D}=C.
\]

We summarize some properties of the degree function in the following lemma. We leave the
elementary proof to the reader.

\begin{lemma}\label{lem:degreefunction}
Let $f,g\co \R^+\to [0,\infty)$ be functions.
\bn
\item[$(1)$] If $f= 0$ is the zero function, then we have 
$\deg_\infty(f)=-\infty$ and $\deg_0(f)=\infty$ and thus $\deg(f)=-\infty-\infty=-\infty$.
\item[$(2)$] If $f$ is monomial in the limit with $d$ and $D$ as in the definition, then $\deg(f)=D-d$. 
\item[$(3)$] If $f=a_rt^r+a_{r+1}t^{r+1}+\dots+a_st^s$ is a polynomial with $a_r\ne 0$ and $a_s\ne 0$, then $\deg(f)=s-r$.
\item[$(4)$] If  one of $f$ or $g$ is monomial in the limit, then $\deg(f\cdot g)=\deg(f)+\deg(g)$.
\item[$(5)$] If $\deg(f)\in \R$, then $\deg(\frac{1}{f})=-\deg(f)$.
\item[$(6)$] If  $f\doteq g$, then $\deg(f)=\deg(g)$. 
\item[$(7)$] If $s\in [0,\infty)$ and if $\deg(f)\in \R$, then $\deg(f^s)=s\deg(f)$. 
\en
\end{lemma}

%==============================================================================
\section{The $L^2$-Alexander torsion for knots} \label{section:l2knots} 
In this section we will study the $L^2$-Alexander torsion for knots, in particular we will relate it to the $L^2$-Alexander invariant that was introduced by Li--Zhang~\cite{LZ06a,LZ06b,LZ08}. We will also prove a relationship between $L^2$-Alexander torsions and the classical Alexander polynomial of a knot.
\smallskip

First recall that given
an oriented knot $K\subset S^3$ we denote by $\nu K$ an open tubular neighborhood of $K$ and that we refer to
$X(K)=S^3\sm \nu K$  as the exterior of $K$. Observe that $X(K)$ is a compact 3-manifold whose boundary consists in a single torus $\partial \nu K$. Furthermore we denote by $\phi_K\in
H^1(X(K);\Z)=\hom(\pi_1(X(K)),\Z)$ the usual \emph{abelianization} which is the epimorphism which sends the oriented meridian to
$1$.  An \emph{admissible homomorphism} is a homomorphism $\g\co \pi_1(X(K))\to G$ such
that $\phi_K$ factors through $\g$. Note that if $\g$ is admissible, then
$(X(K),\phi_K,\g)$ is an admissible triple and we define
\[
\tautwo(K,\g):=\tautwo(X(K),\phi_K,\g)\co \R^+\to [0,\infty).
\] 
If $\g$ is the identity
homomorphism, then we write $\tautwo(K):=\tautwo(K,\g)$ and we refer to $\tautwo(K)$ as
the \emph{full $L^2$-Alexander torsion of $K$}.

It follows from the symmetry of the $L^2$-Alexander torsion, see the discussion preceding
Theorem~\ref{thm:sym}, that these definitions do not depend on the orientation of $K$. We
will henceforth only work with unoriented knots, and given a knot $K$ we mean by $\phi_K$
either one of the two generators of   $H^1(X(K);\Z)=\hom(\pi_1(X(K)),\Z)$. Note that either choice of $\phi_K$ sends the meridian of $K$ to $\pm 1$. It will
not matter which of the two possible choices for $\phi_K$ we take.

In Section~\ref{equ:fox} we show how one can use Fox derivatives in the calculation of
$\tautwo(K,\g)$.  In Section~\ref{section:li-zhang} we will use this calculation to show
that the full $L^2$-Alexander torsion of $K$ is basically the same as the $L^2$-Alexander
invariant of Li--Zhang.  In Section~\ref{section:l2ordinaryalex} we will use $\g=\phi_K$
as the coefficient system and we will see that the resulting $L^2$-Alexander torsion is
determined by the ordinary Alexander polynomial $\Delta_K(z)\in \Z[z^{\pm 1}]$ of $K$.

%==============================================================================
\subsection{Fox derivatives}\label{equ:fox}

In the following we denote by $F$ the free group with generators $g_1,\dots,g_k$.  We then
denote by $\frac{\partial }{\partial g_i}\co \Z[F]\to \Z[F]$ the \emph{Fox derivative with
  respect to $g_i$}, i.e.,  the unique $\Z$-linear map such that
\[
\frac{\partial g_i}{\partial g_i}=1,\quad \frac{\partial g_j}{\partial g_i}=0\mbox{ for $i\ne j$ and } 
\frac{\partial uv}{\partial g_i}=\frac{\partial u}{\partial g_i}+u\frac{\partial v}{\partial g_i} \mbox{ for all  $u,v\in F$.}
\]
We refer to~\cite{Fo53} for the basic properties of the Fox derivatives.

We can now formulate the following lemma which can be viewed as a slight generalization 
of~\cite[Theorem~3.2]{DW10} and~\cite[Theorem~3.5]{DW15}.

\begin{lemma}\label{lem:foxcalc}
  Let $K$ be a knot and let $\pi =\pi_1(X(K))$ denote its group. Consider an admissible
  homomorphism $\g\co \pi\to G$ and let $ \ll g_1,\dots,g_k\,|\, r_1,\dots,r_{k-1}\rr$ be a deficiency one
  presentation for $\pi$. $($We could for example take a Wirtinger presentation for
  $\pi$.$)$ We denote by $B=\big( \frac{\partial r_j}{\partial g_i}\big)$ the $(k-1)\times
  k$--matrix over $\Z[\pi]$ that is given by taking all Fox derivatives of all relations.
  We pick any $i\in \{1,\dots,k\}$ such that $\g(g_i)$ is an element of infinite order. We
  denote by $B_i$ the result of deleting the $i$-th column of $B$. Then we have
  \[
\tautwo(K,\g)\doteq \detr_{\NN(G)}(\kappa(\phi_K,\g,t)(B_i))\cdot
  \max\{1,t^{\phi_K(g_i)}\}^{-1}.
\]
\end{lemma}

\begin{proof}
We write $\phi=\phi_K$. We denote by $Y$ the 2--complex with one 0--cell, $k$ 1--cells
and $k-1$ 2--cells that corresponds to the given presentation. 
The 2--complex $Y$ is   simple homotopy
equivalent to $X(K)$. This statement seems to be well-known, but the only proof we are aware of in the literature is given in~\cite[p.~458]{FJR11}. (The argument in~\cite{FJR11} builds on on the fact that $\pi_1(X(K))$ is locally indicable as proved by Howie~\cite{Ho82} and the fact that the Whitehead group of $\pi$ vanishes, which in turn is a consequence of the Geometrization Theorem.)  The argument at the end of Section~\ref{section:defl2alextorsion}
  shows that we can use $Y$ to calculate $\tautwo(K)$.

  It follows basically from the definition of the Fox derivatives that we can lift the
  cells of $Y$ to the universal cover $\wti{Y}$ such that the chain complex $C_*(\wti{Y}; \Z)$
  is isomorphic to
  \[
\xymatrix@1@+2pc{0\ar[r] &\Z[\pi]^{k-1}\ar[r]^-{B} & \Z[\pi]^k\ar[r]^-{\small \bp 1-g_1\\
    \vdots\\1-g_k\ep}& \Z[\pi]\ar[r] & 0}.
\] 
  Again we refer to~\cite{Fo53} for details.  This
  implies that for any $t\in \R^+$ the chain complex $\R[G]\otimes_{\Z[\pi]}C_*(\wti{Y})$
  is isomorphic to
  \[
\xymatrix@1@+4.5pc{0\ar[r] &  \R[G]^{k-1}\ar[r]^-{\kappa(\phi,\g,t)(B)} & \R[G]^k\ar[r]^-{\small \bp
    1-t^{\phi(g_1)}\g(g_1)\\ \vdots\\ 1-t^{\phi(g_k)}\g(g_k)\ep}& \R[G]\ar[r] & 0}.
\] 
  Since  $\g(g_i)$ has infinite order it follows from Lemma~\ref{lem:det1minusg} that
  $\detr_{\NN(G)}\left(1-t^{\phi(g_i)}\g(g_i)\right)=\max\{1,t^{\phi(g_i)}\}$.  The lemma
  now follows immediately from Lemma~\ref{lem:torsion2complex}.
\end{proof}

%==============================================================================
\subsection{The $L^2$-Alexander invariant of Li--Zhang}\label{section:li-zhang}
Let $K$ be a knot. We pick a Wirtinger presentation $ \ll g_1,\dots,g_k\,|\,
r_1,\dots,r_{k-1}\rr$ for $\pi=\pi_1(X(K))$.  We denote by $B=( \frac{\partial
  r_j}{\partial g_i})$ the $(k-1)\times k$--matrix over $\Z[\pi]$ that is given by taking
all Fox derivatives of all relations.  We pick any $i\in \{1,\dots,k\}$ and we denote by
$B_i$ the result of deleting the $i$-th column of $B$. The \emph{$L^2$-Alexander
  invariant $\Delta_K^{(2)}$ of $K$} is then defined as the function
\[
\ba{rcl} \Delta_K^{(2)}\co \C\sm \{0\}&\to &[0,\infty) \\
t&\mapsto & \detr_{\NN(G)}\big(\kappa(\phi,\g,t)(B_i)\big).\ea 
\] 
This invariant was first
introduced by Li--Zhang~\cite[Section~7]{LZ06a}
and~\cite[Section~3]{LZ06b}, using slightly different conventions. In these papers it is also implicitly proved that the
function $\Delta_K^{(2)}$, as an invariant of $K$, is well-defined up to multiplication by
a function of the form $t\mapsto |t|^{n}$ with $n\in \Z$. Furthermore,
Li--Zhang~\cite{LZ06a,LZ06b} implicitly, and Dubois--Wegner~\cite[Proposition~3.2]{DW15}
explicitly showed that for any $t\in \C\sm \{0\}$ we have
$\Delta_K^{(2)}(t)=\Delta_K^{(2)}(|t|)$. To be consistent with our other conventions we
henceforth view $\Delta_K^{(2)}$ as a function defined on $\R^+$.
(Recall that at the end of Section~\ref{section:defl2alextorsion} we already remarked that the aforementioned result of Li--Zhang and Dubois--Wegner was the reason why we view $\tautwo(N,\phi,\gamma)$ as a function on $\R^+$, even though a priori one could also view it as a function on $\C\sm \{0\}$.)
 
It now follows from Lemma~\ref{lem:foxcalc} and the fact that every generator of a
Wirtinger presentation is a meridian that
\[
\tautwo(K) \doteq \Delta_K^{(2)}\cdot \max\{1,t\}^{-1}.
\] 
This shows that the full
$L^2$-Alexander torsion and the $L^2$-Alexander invariant are essentially the same
invariant.

%==============================================================================
\subsection{The $L^2$-Alexander torsion and the one-variable Alexander polynomial}\label{section:l2ordinaryalex}
We will now see that given a knot $K$ the ordinary Alexander polynomial $\Delta_K(z)\in
\Z[z^{\pm 1}]$ determines the $L^2$-Alexander torsion corresponding to the abelianization.
More precisely, we have the following proposition.

\begin{proposition}\label{prop:l2ordinaryalex}
Let $K$ be a knot and let $\Delta_K(z)\in \Z[z^{\pm 1}]$ be a representative of the Alexander polynomial of $K$.
We write
\[
\Delta_K(z)=C\cdot z^m\cdot \prod_{i=1}^k (z-a_i),
\]
where $C\in \Z\sm \{0\}, m\in \Z$ and $a_1,\dots,a_k\in\C\sm \{0\}$.
Then
\[
\tautwo(K,\phi_K)\doteq C\cdot \prod_{i=1}^k \max\{|a_i|,t\}\cdot \max\{1,t\}^{-1}.
\]
\end{proposition}

Note that the proposition can also be proved using Proposition~\ref{prop:l2abelian} and
the fact that the Reidemeister torsion of a knot corresponding to the abelianization
equals $\Delta_K(z)\cdot (z-1)^{-1}$, see~\cite{Tu01} for details.

\begin{proof}
  Let $K$ be a knot. We write $\phi=\phi_K\in H^1(X(K);\Z)=\hom(\pi_1(X(K)),\ll z\rr)$. Let $ \ll g_1,\dots,g_k\,|\,
  r_1,\dots,r_{k-1}\rr$ be a Wirtinger presentation for $\pi=\pi_1(X(K))$.  Again we
  denote by $B=( \frac{\partial r_j}{\partial g_i})$ the $(k-1)\times k$--matrix over
  $\Z[\pi]$ that is given by taking all Fox derivatives of all relations.  We pick any
  $i\in \{1,\dots,k\}$ and we denote by $B_i$ the result of deleting the $i$-th column of
  $B$.

  Now we apply the ring homomorphism $\phi\co \Z[\pi]\to \Z[\ll z\rr]=\zz$ to all entries
  of $B_i$ and we denote the resulting matrix by $A_i(z)$.  Note that
  by~\cite[Chapter~VIII.3]{CF63} we have $\det(A_i(z))=\Delta_K(z)$.

  Given $t\in \R^+$ we denote by $A_i(tz)$ the matrix over $\rz$ that is given by
  substituting $z$ by $tz$. Similarly we denote by $\Delta_K(tz)$ the polynomial over
  $\rz$ which is given by substituting $z$ by $tz$.  It follows immediately from the
  definitions that $\kappa(\phi,\phi,t)(B_i)=A_i(tz)$. Also note that
  $\det(A_i(tz))=\Delta_K(tz)$.

By  Lemmas~\ref{lem:foxcalc} and  by the discussion in Section~\ref{section:mahlersub} we  have 
\[
\ba{rcl} \tautwo(K,\phi_K)(t)&\doteq &\detr_{\NN(\ll z\rr)}(A_i(tz))\cdot \max\{1,t\}^{-1}\\
&=&m(\det(A_i(tz))\cdot \max\{1,t\}^{-1}\\
&=&m(\Delta_K(tz))\cdot \max\{1,t\}^{-1}\\
&=&m\left(C\cdot (zt)^m\prod_{i=1}^k(tz-a_i)\right)\cdot \max\{1,t\}^{-1}\\[2mm]
&\doteq &m\left(C\cdot (zt)^m\prod_{i=1}^k(z-a_it^{-1})\right)\cdot \max\{1,t\}^{-1}\\
&\doteq&C\cdot \prod_{i=1}^k \max\{1,t^{-1}\cdot |a_i|\}\cdot \max\{1,t\}^{-1}\\[2mm]
&\doteq&
 C\cdot \prod_{i=1}^k \max\{t,|a_i|\}\cdot \max\{1,t\}^{-1}.\ea 
\]
\end{proof}

We obtain the following corollary:

\begin{corollary}\label{cor:ordinaryalex}
Given any knot $K$ the $L^2$-Alexander torsion $\tautwo(K,\phi_K)$ is a piecewise monomial function with
\[
\deg \tautwo(K,\phi_K)=\deg \Delta_K(t).
\]
Furthermore $\tautwo(K,\phi_K)$ is monic if and only if $\Delta_K(t)$ is monic.
\end{corollary}

\begin{proof}
Let $\Delta(z)=\Delta_K(z)\in \Z[z^{\pm 1}]$ be a representative of the Alexander polynomial of $K$.
We write $ \Delta(z)=C\cdot z^m\cdot \prod_{i=1}^k (z-a_i)$,
where $C\in \Z\sm \{0\}, m\in \Z$ and $a_1,\dots,a_k\in\C$ such that $|a_1|\leq |a_2|\leq \dots \leq |a_k|$.
By  Proposition~\ref{prop:l2ordinaryalex}
we have
\[
\tautwo(K,\phi_K)\doteq \t(t):=C\cdot \prod_{i=1}^k \max\{ |a_i|,t\}\cdot \max\{1,t\}^{-1}.
\]
It follows immediately that $\tautwo(K,\phi_K)$ is a piecewise polynomial.
Note that
\[
\ba{rcll} \t(t)&=&C\cdot t^k, &\mbox{for }t\geq \max\{1,|a_1|,\dots,|a_k|\}, \mbox{ and }\\
\t(t)&=&C\cdot \prod_{i=1}^k |a_i|, &\mbox{for }t\leq \min\{1,|a_1|,\dots,|a_k|\}.\ea 
\]
Thus we see that
\[
\deg \tautwo(K,\phi_K)=\deg \t=k.
\]
It is well--known that the Alexander polynomial is symmetric, i.e.,  $\Delta_K(z)=z^{l}\Delta_K(z^{-1})$ for some $l\in \Z$.
It follows in particular that the set of zeros is closed under inversion, i.e.
$\{a_1,\dots,a_k\}=\{a_1^{-1},\dots,a_k^{-1}\}$ as a set with multiplicities.
This implies  that $\left|\prod_{i=1}^ka_i\right|=1$.
It now follows that  $\tautwo(K,\phi_K)$ is monic if and only if $\Delta_K(z)$ is monic.
\end{proof}

\begin{remark}
Proposition~\ref{prop:l2ordinaryalex} shows that the $L^2$-Alexander torsion $\tautwo(K,\phi_K)$ 
contains a lot of the essential information of the ordinary Alexander polynomial $\Delta_K(t)$. 
Nonetheless, some information gets lost.
For example, let $K=T_{p,q}$ be the $(p,q)$-torus knot. It is well-known, see e.g.\ \cite{Ro90}, that 
\[
\Delta_{T_{p,q}}(t)=\frac{(t^{pq}-1)(t-1)}{(t^p-1)(t^q-1)}.
\] 
This is a polynomial of
degree $(p-1)(q-1)$ and all the zeros are roots of unity.  It thus follows from
Proposition~\ref{prop:l2ordinaryalex} that
\[
\tautwo(T_{p,q},\phi_K)=\max\{1,t\}^{(p-1)(q-1)-1}.
\] 
(In fact we will see in
Theorem~\ref{thm:iteratedtorus} that this equality holds for any admissible epimorphism
$\g$.) In particular, if we consider the torus knots $T_{3,7}$ and $T_{4,5}$, then it is
now straightforward to see that all $L^2$-Alexander torsions agree, but that the ordinary
Alexander polynomials are different.
\end{remark}

%==============================================================================
\section{Calculations of  $L^2$-Alexander torsions for special classes of 3-manifolds}
\label{section:topological-info}
In this section we first give a complete calculation of the $L^2$-Alexander torsion for graph manifolds which allows us to reprove the fact that the $L^2$-Alexander torsion detects the unknot. Then we give a partial calculation of the $L^2$-Alexander torsion for fibered classes.

%==============================================================================
\subsection{$L^2$-Alexander torsions of graph manifolds}\label{section:graph}
First we recall that a graph manifold is a 3-manifold for which all its JSJ components are
Seifert fibered spaces. The following theorem gives the computation of the $L^2$-Alexander
torsions of Seifert fibered spaces. 
 The proof of the theorem builds on \cite[Theorem~3.105]{Lu02} and the
details can be found in \cite{Her15}.

\begin{theorem}\label{thm:sfs}
  Let $(N,\phi,\g)$ be an admissible triple with $N\ne S^1\times D^2$ and $N\ne S^1\times
  S^2$. Suppose that $N$ is a Seifert fibered 3-manifold such that the image of a regular
  fiber under $\g$ is an element of infinite order, then
  \[
\tautwo(N,\phi,\g)\doteq \max\{1,t^{x_N(\phi)}\}.
\]
\end{theorem}

Now we obtain the following result which is a slight refinement of Theorem~\ref{thm:graphintro}.

\begin{theorem}\label{thm:graph}
  Let $(N,\phi,\g)$ be an admissible triple with $N\ne S^1\times D^2$ and $N\ne S^1\times
  S^2$. Suppose that $N$ is a graph manifold and that given any JSJ component of $N$ the
  image of a regular fiber under $\g$ is an element of infinite order, then
  \[
\tautwo(N,\phi,\g)\doteq \max\{1,t^{x_N(\phi)}\}.
\] 
  In particular
  $\tautwo(N,\phi,\g)$ is monomial in the limit with degree $x_N(\phi)$ and furthermore
  $\tautwo(N,\phi,\g)$ is monic.
\end{theorem}

\begin{proof}
  We denote by $N_i, i=1,\dots,k$ the JSJ components of $N$, which by assumption are
  Seifert fibered spaces. Note that for each $i$ we have $N_i\ne S^1\times D^2$ and
  $N_i\ne S^1\times S^2$.  For $i=1,\dots,k$ we write $\phi_i=\phi|_{N_i}$ and we write
  $\g_i=\g|_{\pi_1(N_i)}$.  By our assumption on $\g$ and by Theorem~\ref{thm:sfs} we
  have $\tau(N_i,\phi_i,\g_i)\doteq \max\{1,t^{x_{N_i}(\phi_i)}\}$.  Furthermore, note
  that the Seifert fibered structure of a Seifert fibered 3-manifold restricts to a
  fibration of any boundary torus. It  follows from our assumption on $\g$ that the
  restriction of $\g$ to any JSJ torus has infinite image. Thus it follows from
  Theorem~\ref{thm:jsj} that
  \[
\tautwo(N,\phi,\g)\doteq \prod_{i=1}^k \tautwo(N_i,\phi,\g_i)\doteq
  \prod_{i=1}^k\max\{1,t^{x_{N_i}(\phi_i)}\}=\max\left\{1,t^{\sum_{i=1}^kx_{N_i}(\phi_i)}\right\}.
\]
  The theorem follows from~\cite[Proposition~3.5]{EN85} which says in our situation  that 
  \[ \sum_{i=1}^k
  x_{N_i}(\phi_i)=x_N(\phi).\]
\end{proof}

%==============================================================================
\subsection{Applications to knot theory}
\label{section:applications-to-knot-theory}
We denote by $\KK$ the minimal set of oriented knots that contains the unknot and that is
closed under the connect sum operation and under cabling.  Note that $\KK$ contains torus
knots, and more generally iterated torus knots.  We recall the following well-known lemma.

\begin{lemma}\label{lem:go83} 
Let $K$ be a knot. The following statements are equivalent:
\bn
\item[$(1)$] $K$ lies in $\KK$.
\item[$(2)$] the knot exterior $X(K)=S^3\sm \nu K$ is a graph manifold with the property that the regular fiber 
of any Seifert fibered piece  is  non-zero in $H_1(X(K);\Z)$.
\item[$(3)$]  $X(K)$ is a graph manifold.
\en
\end{lemma}

Here the implication (1) $\Rightarrow$ (2) is not hard to verify.  The implication (2) $\Rightarrow$
(3) is trivial and the implication (3) $\Rightarrow$ (1) is~\cite[Corollary~4.2]{Go83}.

We can now state and prove the following theorem.

\begin{theorem}\label{thm:iteratedtorus}
Let $K$ be a knot in the set $\KK$. For any admissible epimorphism $\g\colon \pi_1(X(K))\to G$  we have
\[
\tautwo(K,\g)\doteq \max\left\{1,t^{2\,\genus(K)-1}\right\}.
\]
In particular, if $K=T_{p,q}$ is the $(p,q)$-torus knot, then 
\[
\tautwo(T_{p,q},\g)\doteq \max\left\{1,t^{(p-1)(q-1)-1}\right\}.
\]
\end{theorem} 

\begin{proof}
  First note that if $K$ is the trivial knot, then $X(K)=S^1\times D^2$.  In this case we
  identify $\pi_1(X(K))$ with the infinite cyclic group generated by $\mu$.  Since $X$ is
  simple homotopy equivalent to a circle it follows from the definitions and from
  Lemma~\ref{lem:det1minusg} that
\[
\tautwo(K,\g)\doteq \tautwo\left( \xymatrix@1{
 0\ar[r] & \R[\ll \mu\rr] \ar[r]^-{1-t\mu} & \R[\ll \mu\rr] \ar[r] & 0 } \right)\doteq \max\{1,t\}^{-1}.
\]

Now let $K$ be a non-trivial knot. It is well-known and straightforward to show that in this case the equality
$x_{X(K)}(\phi_K)=2\,\mbox{genus}(K)-1$ holds.  Now we suppose that $K$ lies in $\KK$. By
Lemma~\ref{lem:go83} the knot exterior $X(K)$ is a graph manifold with the property that
the regular fiber of any Seifert fibered piece represents is non-zero in $H_1(X(K);\Z)$.
Since $\g$ is admissible we can appeal to Theorem~\ref{thm:graph} to obtain the desired
result.  The statement for torus knots follows from the well-known fact that the genus of
the $(p,q)$-torus knot is $\frac{1}{2}(p-1)(q-1)$.
\end{proof}

The combination of Theorems~\ref{thm:ls99} and~\ref{thm:iteratedtorus} and 
Lemma~\ref{lem:go83} immediately implies Theorem~\ref{thm:detectsunknot}.

%==============================================================================
\subsection{Fibered classes and the $L^2$-Alexander torsion}\label{section:fibered}
Let $G$ be a group with finite generating set $S$. Given $g\in G$ we denote by $\ell_S(g)$
the minimal length of a word in $S$ representing $g$. In this paper the \emph{entropy} of a homomorphism
$f\co G\to G$ is defined as
\[
h(f):=\max\Big\{\limsup_{n\to
      \infty}\,\big(\ell_S(f^n(g))\big)^{\frac{1}{n}}\,\Big|\,g\in S\Big\}.
\] 
Note that the entropy is independent of the choice of $S$. We refer to~\cite[p.~185]{FLP79} for
details.  (Note though that  we take the exponential of the
entropy as defined in~\cite{FLP79}.)

Now let $\S$ be a surface and let $f\co \S\to \S$ be a self-diffeomorphism.
 Choose $x \in \Sigma$ and a path $w$ in $X$ from $f(x)$
  to $x$.  Define the entropy $h(f)$ of $f$ to be the entropy of the group automorphism
\[\alpha(f,x,w) \colon \xymatrix@1{\pi_1(\Sigma,x) \ar[r]^-{\pi_1(f,x)} & \pi_1(\Sigma,f(x)) \ar[r]^-{t_w} & \pi_1(\Sigma,x)},\] where $t_w$ is given by conjugation with the path $w$.
  One easily checks that this definition is independent of the choice of $x$ and $w$ and the entropy $h(f)$
depends only on the homotopy class of $f$.
\begin{remark} 
If $\Sigma$ is a closed surface
 with $\chi(\S)<0$  and if $f$ is pseudo-Anosov, then 
by~\cite[p.~195]{FLP79} the entropy $h(f)$ equals the  dilatation of $f$.
\end{remark}

Given a 3-manifold $N$ and a primitive fibered class $\phi$ we define 
$h(\phi)$ as the entropy of the corresponding monodromy. More generally, for a fibered class $\phi\in H^1(N;\Q)$ we pick an $r\in \Q_{>0}$ such that $r\phi$ is a primitive integral class and we define
\[ h(\phi):=h(r\phi)^{\frac{1}{r}}.\]

Before we state the next theorem, recall that in Section~\ref{section:classg} we said that  $\GG$ denotes the class of all sofic groups. We also mentioned that the class $\GG$ is known to contain practically all groups we are interested in, in particular the fundamental groups of 3-manifolds.

Now we have the following theorem which is a generalization of Theorem~\ref{thm:l2fiberedintro} in the introduction.

\begin{theorem}\label{thm:l2fibered}
  Let $(N,\phi,\g)$ be an admissible triple  with $N\ne S^1\times D^2$ and $N \ne S^1\times S^2$  such that $\phi\in H^1(N;\Q)$ is fibered and
  such that $G\in \GG$. There exists a representative $\tau$ of $\tautwo(N,\phi,\g)$
  such that 
  \[
\tau(t)=\left\{ \ba{ll} 1, &\mbox{ if }t<\tmfrac{1}{h(\phi)}, \\ t^{x_N(\phi)}, &\mbox{ if
    }t>h(\phi).\ea \right.
\] 
In particular $\tautwo(N,\phi,\g)$ is monomial in the limit with
  degree $x_N(\phi)$ and furthermore $\tautwo(N,\phi,\g)$ is monic.
\end{theorem}

The proof of Theorem~\ref{thm:l2fibered} will require the remainder of
Section~\ref{section:fibered}. The key ingredient from the theory of Fuglede--Kadison determinants is a  theorem of
Caray--Farber--Mathai~\cite{CFM97}.  In order to state the theorem we need to introduce a few more definitions. Given a group $G$ we denote by
$\gl(n,\nng)$ the group of invertible $n\times n$--matrices with entries in $\nng$.
Secondly, if $f\in \nng$ then we write $ \op{tr}_{G}(f):=\ll f(e),e\rr_{l^2(G)}$, where $e\in l^2(G)$
denotes the unit element and $\ll-,-\rr_{l^2(G)}$ denotes the inner product on $l^2(G)$. Furthermore, if $A=(a_{ij})$ is an $n\times n$--matrix over
$\nng$, then we define
\[
\tr_G(A):=\sum_{i=1}^n \tr_G(a_{ii}).
\]
We can now formulate the following theorem of Caray--Farber--Mathai~\cite[Theorem~1.10~(e)]{CFM97}.

\begin{theorem}\label{thm:cfm97}
Let $G$ be a group, let $t\in \R^+$ and let
\[
\ba{rcl} A\colon [0,t]&\to & \gl(n,\nng)\\
s&\mapsto &A(s)\ea 
\]
be a continuous piecewise smooth map, then
\[
\detr_{\NN(G)}(A(t))=\detr_{\NN(G)}(A(0))\cdot \exp\left(\int_0^t \op{Re}\,\op{tr}_{G}\left(A(s)^{-1} \cdot \smfrac{d}{ds}A\big|_s\right)\,ds\right).
\]
\end{theorem}

Let $G$ be a group.
Before we continue we need to introduce a norm of matrices over the group ring $\R[G]$.
First given $p=\sum_{g\in G} a_gg\in \R[G]$ we write 
\[
|p|_1:=\sum_{g \in G} |a_g|
\]
and given an $n\times n$-matrix $A=(a_{ij})$ over $\R[G]$ we  write
\[
\|A\|_1:=n\cdot \max\big\{ |a_{ij}|_1\,\,\big|\,\, i,j=1,\dots,n\big\}
\]
and we define
\[
h(A):=\lim_{k\to \infty} \big(\|A^k\|_1\big)^{\frac{1}{k}}.
\] 
The existence of the limit hereby follows from Fekete's subadditive lemma and the following elementary lemma.

\begin{lemma}\label{lem:norm-submultiplicative}
Let $A$ and $B$ be two $n\times n$-matrices over $\R[G]$. Then
\[ \| AB\|_1\,\,\leq \,\,\|A\|_1\cdot \|B\|_1.\]
\end{lemma}

\begin{proof}
First, by \cite[Lemma~2.1.5~on~p.~35]{Pa77}  given any $p,q\in \R[G]$
we have the inequality
\[ | pq|_1\,\,\leq \,\,|p|_1\cdot |q|_1.\]
The desired inequality for the matrices now follows from this inequality and the fact that any entry of $AB$ is a sum of at most $n$ products of entries in $A$ and $B$.
\end{proof}

\begin{proposition} \label{prop:basiccaseofdetpq} Let $G$ be a group in $\GG$, let $\phi\co G\to
  \Z$ be an epimorphism, and let $y\in G$ be an element  with $\phi(y)=1$.  We
  write $H=\ker(\phi)$.  Let $P,Q$ be two $n\times n$-matrices over $\Z[H]$   which are invertible over $\Z[H]$. Then for $T=h(yQP^{-1})$,  we have
  \[
\detr_{\NN(G)}(P-ty Q)=\left\{ \ba{ll} 1, &\mbox{ if }t<\tmfrac{1}{T}, \\ t^n, &\mbox{
      if }t>T.\ea \right.
\]
\end{proposition}

\begin{proof}
  We write $A=QP^{-1}$ and $T=h(yA)$. Since $G\in \GG$ it follows from 
 Proposition~\ref{prop:detl2square} and Theorem~\ref{thm:sofic} that
\[
\detr_{\NN(G)}(P-ty Q)=\detr_{\NN(G)}(\id-ty QP^{-1})\cdot \detr_{\nng}(P)=\detr_{\NN(G)}(\id-ty A).
\]
Now let $s\in (0,\frac{1}{T})$. It follows easily from the definition of $h(yA)$ and from Lemma~\ref{lem:norm-submultiplicative}  that the power series
\[
\sum_{i=0}^\infty s^i(yA)^i
\]
converges  in the operator norm and that it is an inverse to $P(s) = \id - syA$.

For any $t\in (0,\frac{1}{T})$ we can thus apply Theorem~\ref{thm:cfm97} and we obtain that
\[
\ba{rcl} \detr_{\NN(G)}(\id-tyA)&=&\detr_{\NN(G)}(P(t))\\
&=&\detr_{\NN(G)}(P(0))\cdot \exp\left(\int_0^1\op{Re}\,\tr_{G}\left(P(s)^{-1}\frac{d}{ds}P|_s\right)\ds\right)\\[2mm]
&=&\exp\left(\int_0^1\op{Re}\,\tr_{G}\left( \left(\sum_{i=0}^\infty (tyA)^i\right) (-yA)\right)\dt\right)\\[2mm]
&=&\exp\left(\int_0^1\op{Re}\,\tr_{G}\left( \sum_{i=0}^\infty -t^i(yA)^{i+1}\right)\dt\right).\ea
\]
Note that
\[
\tr_{G}\bigg( \sum_{i=0}^\infty -t^i(yA)^{i+1}\bigg)=\sum_{i=0}^\infty \tr_{G}(-t^i(yA)^{i+1}).
\]
Also note that any entry of $(yA)^{i+1}$ is of the form
\[
\sum_{j=1}^l a_jg_j\mbox{ with }a_1,\dots,a_l\in \Z \mbox{ and } g_1,\dots,g_l\in G,
\]
where $\phi(g_1)=\dots=\phi(g_l)=i+1$. It follows immediately that $\tr_{G}(-t^i(yA)^{i+1})=0$ for all $i\geq 0$.
Thus we see that
\[ 
\detr_{\NN(G)}(\id-tyA)=1\mbox{ for all }t\in \big(0,\tmfrac{1}{T}\big).
\]
Now suppose that $t>T$. It follows from Theorem~\ref{thm:sofic} and from the above that  
\[
\ba{rcl} \detr_{\NN(G)}(\id-ty A)&=&\detr_{\NN(G)}\big(tAy(t^{-1}y^{-1}A^{-1}-\id)\big)\\[2mm]
&=&t^n\,\,\detr_{\NN(G)}(Ay)\,\,\detr_{\NN(G)}(t^{-1}y^{-1}A^{-1}-\id)=t^n.\ea
\]
\end{proof}

Now we are  finally in a position to prove Theorem~\ref{thm:l2fibered}.

\begin{proof}[Proof of Theorem~\ref{thm:l2fibered}]
Let $(N,\phi,\g\colon \pi_1(N)\to G)$ be  an admissible triple  with $N\ne S^1\times D^2$ and $N \ne S^1\times S^2$ such that $\phi\in H^1(N;\Q)$ is fibered and
such that $G\in \GG$. By Lemma~\ref{lem:multiple} and by the definition of the entropy of a rational fibered class we only need to consider the case that $\phi$ is a primitive.

 We denote by $\Sigma$ the fiber and we denote by $f\colon \Sigma\to \Sigma$ the monodromy corresponding to the primitive fibered class $\phi$. 
 
If $\chi(\Sigma)\geq 0$, then $x_N(\phi)=0$ and $N$ is a graph manifold.
Thus the statement follows immediately from the calculation of the $L^2$-Alexander torsion for graph manifolds given in  Theorem~\ref{thm:graph}.

For the remainder of this paper we  assume that $\chi(\Sigma)<0$.
Since $\chi(\Sigma)<0$ there exists a fixed point $p\in \Sigma$ of the monodromy $f$. 
We pick a CW-structure for $\Sigma$ with one 0-cell $p$, $n$ 1-cells
$g_1,\dots,g_n$ and one 2-cell which by a slight abuse of notation we denote again by
$\Sigma$.  By another slight abuse of notation we denote the elements in
$\pi_1(\Sigma,p)$ represented by $g_1,\dots,g_n$ by the same symbols.
It is well-known,   see e.g.\ \cite[Theorem~3]{Th86}, 
that a fiber surface is Thurston norm minimizing, in particular, in our context this means that  $n-2=-\chi(\Sigma)=x_N(\phi)$.

In the following we compute the entropy using the  generating set $S=\{g_1,\dots,g_n\}$ of $\pi_1(\Sigma,p)$. By definition we have 
\[ h(f) =\max\Big\{\limsup_{m\to
      \infty}\,\big(\ell_S(f^m_*(g_i))\big)^{\frac{1}{m}}\,\Big|\,i=1,\dots,n\Big\},\]
where $f_*$ denotes the induced map on $\pi_1(\Sigma,p)$.
 We write $T=h(f)$. We will prove the following claim.

\begin{claim}
For any $\eps>0$ there exists a representative $\tau$ of $\tautwo(N,\phi,\g)$ such that 
  \[
\tau(t)=\left\{ \ba{ll} 1, &\mbox{ if }t\in \big(0,\frac{1}{T+\eps}\big), \\ t^{n-2}, &\mbox{ if
    }t>T+\eps.\ea \right.
\] 
\end{claim}

This claim implies Theorem~\ref{thm:l2fibered}.
Indeed, we already mentioned that $n-2=\chi_N(\phi)$. Furthermore any two representatives of $\tautwo(N,\phi,\g)$ that coincide at some point $t\ne 1$ are necessarily the same. Put differently, the representatives for each $\eps>0$ are in fact always the same representative of $\tautwo(N,\phi,\g)$.

Now we turn to the proof of the claim. Let $\eps>0$. 
By the definition of $h(f)$ there exists an $m\in \N$ such that   
\[ \big(\ell_S(f^m_*(g_i))\big)^{\frac{1}{m}}<h(f)+\eps\]
for $i=1,\dots,n$.

We denote by $p\colon \what{N}\to N$ the $m$-fold cyclic cover of $N$  corresponding to the subgroup $\what{\pi}:=\ker\{\pi_1(N)\xrightarrow{\phi}\Z\to \Z_m\}$. We write $\what{\phi}:=p^*\phi$ and we denote by
  $\what{\g}$ the restriction of $\g$ to $\what{\pi}$.
By Lemma~\ref{lem:finitecover} we have 
  \[
\tau^{(2)}(\what{N},\what{\phi},\what{\g})(t)\doteq
  \left(\tautwo(N,\phi,\g)(t)\right)^{m}.
\]
We write $\psi=\frac{1}{m}\what{\phi}$. Note that $\psi$ is a primitive fibered class of $\what{N}$. By Lemma~\ref{lem:multiple} we have
\[
\tau^{(2)}(\what{N},\psi,\what{\g})(t)\doteq
\tau^{(2)}(\what{N},\what{\phi},\what{\g})(t^{1/m}).\]
Putting everything together we see that now it suffices to prove the following claim.

\begin{claim}
There exists a representative $\tau$ of $\tau^{(2)}(\what{N},\psi,\what{\g})(t)$ such that 
  \[
\tau(t)=\left\{ \ba{ll} 1, &\mbox{ if }t\in \big(0,\frac{1}{T+\eps}\big), \\ t^{n-2}, &\mbox{ if
    }t>T+\eps.\ea \right.
\] 
\end{claim}

  Given a  map $r\co \Sigma\to \Sigma$ we denote by
  \[
M(\Sigma,r):=\Sigma\times [-1,1]\,\,/\,\,(x,-1)\sim (r(x),1)
\] 
  the corresponding mapping torus.
  Note that $M(\Sigma,r)$ has a canonical projection map $M(\Sigma,r)\to S^1=[-1,1]/\{-1\}\sim \{1\}$ and we
  refer to the induced epimorphism $\pi_1(M(\Sigma,r))\to \pi_1(S^1)=\Z$ as the canonical epimorphism
  to $\Z$.  It is clear  that two homotopic maps $r_0,r_1\co \Sigma\to \Sigma$ give rise to
   homotopy equivalent mapping
  tori. 

Recall that we denote by $f\colon \Sigma\to \Sigma$ the monodromy corresponding to $\phi$. This means that we can identify $N$ with $M(\Sigma,f)$  in such a way
that $\phi\in H^1({N};\Z)=\hom(\pi_1({N}),\Z)$ agrees with the canonical epimorphism
$\pi_1(M(\Sigma,f))\to \Z$. By the standard theory of covering spaces of fibered manifolds we can then also identify the $m$-fold cyclic cover $\what{N}$  of $N$ with $M(\Sigma,f^m)$  in such a way
    that the primitive class $\psi=\frac{1}{m}\what{\phi}\in H^1(\what{N};\Z)=\hom(\pi_1(\what{N}),\Z)$ agrees with the canonical epimorphism
    $\pi_1(M(\Sigma,f^m))\to \Z$.

   By the Cellular
  Approximation Theorem the diffeomorphism $f^m$ is homotopic to a cellular map $s$.  In
  fact one can see `by hand' that $c$ can be chosen such that $s(p)=p$ and such that each
  $s(g_i)$ is represented by the path traced out by the word $f_*^m(g_i)$ in the generators
  $g_1,\dots,g_n$.

Now we write $I=[-1,1]$. Given a cell $c$ of $\Sigma$ we denote by $c\times I$ the corresponding product
  cell of $\Sigma\times I$.  Furthermore we denote by $\mu$ the element in
  $\what{\pi}=\pi_1(\what{N})=\pi_1(M(\Sigma,s),p)$ that is represented by the loop $p\times I$.  Note
  that the product CW-structure on $\Sigma\times I$ descends to a CW-structure on
  $M(\Sigma,s)$.  By the above we know that $\what{N}=M(\Sigma,f^m)$ is homotopy equivalent to
  $M(\Sigma,s)$. Since the Whitehead group of fibered 3-manifolds is trivial,
  see~\cite{Wa78}, these two spaces are in fact simple homotopy equivalent.  Thus we have
  $ \tau^{(2)}(\what{N},\psi,\what{\g})(t)\doteq \tautwo(M(\Sigma,s),\psi,\what{\g})(t)$.

%We write $M:=M(\Sigma,s)$ and we identify the fundamental groups  $\pi:=\pi_1(N)$ and $\pi_1(M)$. We then have

Now we collect the cells of $M(\Sigma,s)$  according to their dimensions and we order them as follows:
\[
\{ \Sigma \times I\}\,\, \{ \Sigma, g_1\times I,\dots,g_n\times I\}\,\,\{g_1,\dots,g_n,p\times I\}\,\,\{p\}.
\]
It is straightforward to see that for an appropriate lift of the above ordered sets of
cells of $M(\Sigma,s)$ to the universal cover, the resulting chain complex of the
universal cover is isomorphic to
\[
\xymatrix@1{0\ar[r] & \Z[\what{\pi}] \ar[r]^-{B_3} & \Z[\what{\pi}] \oplus \Z[\what{\pi}]^n\ar[r]^-{B_2} & \Z[\what{\pi}]^n\oplus \Z[\what{\pi}] \ar[r]^-{B_1} & \Z[\what{\pi}] \ar[r] & 0}
\]
where 
\[
B_3=\bp 1-\mu &*\dots &*\ep, \,\, B_2=\bp *&* \\ \id_n-\mu A & * \ep, \,\, B_1=\bp *\\
1-\mu\ep,
\] and where in turn the $(i,j)$-entry of the $n\times n$-matrix $A$ is given by
$\frac{\partial (f^m)_*(g_i)}{\partial g_j}$ and where the $*$'s indicate matrices of an
appropriate size. Note that each entry of the $i$-th row of $A$ is a sum of at most
$\ell_S(f_*^m(g_i))$ elements in $\{g_1^{\pm 1},\dots,g_n^{\pm 1}\}$, possibly equipped with a minus sign. It thus follows
immediately from the definitions and our choice of $m$ that
\[ \| A\|_1\leq h(f)+\eps.\]
Clearly $\|\mu A\|_1=\|A\|_1$. Thus we obtain that 
\[ \| \mu A\|_1\leq h(f)+\eps.\]
By Lemma~\ref{lem:norm-submultiplicative} we also have $ \|(\mu A)^k\|_1\leq \|\mu A\|_1^k$
for any $k\in \N$.
Putting everything together this implies that  
\be \label{equ:ha-in-terms-of-entropy}
h(\mu A)=\lim_{k\to \infty} \big(\|(\mu A)^k\|_1\big)^{\frac{1}{k}}\leq \|\mu A\|_1\leq h(f)+\eps.\ee
By Lemma~\ref{lem:det1minusg} we have $\detr_{\nng}(1-t\what{\g}(\mu))^{-1}=\max\{1,t\}$. Therefore by the
definitions and by Lemma~\ref{lem:torsion3complex} we  have
\[
\ba{rcl}&& \tau(\what{N},\psi,\what{\g})\\[2mm]
&\doteq &\tautwo\Big( \xymatrix@1@+2.9pc{\R[G] \ar[r]^-{\kappa(\psi,\what{\gamma},t)(B_3)} & \R[G] \oplus \R[G]^n
\ar[r]^-{\kappa(\psi,\what{\gamma},t)(B_2)} & \R[G]^n\oplus \R[G]\ar[r]^-{\kappa(\psi,\what{\gamma},t)(B_1)}& \R[G]}\Big)\\[2mm]
&=&\detr_{\nng}\big(1-t\what{\g}(\mu)\big)^{-1}\,\cdot\, \detr_{\nng}\big(\id-t\what{\g}(\mu)\what{\g}(A)\big) \,\cdot \,  \detr_{\nng}\big(1-t\what{\g}(\mu)\big)^{-1}\\[2mm]
&=&\max\{1,t\}^{-2}\,\cdot \, \det_{\nng}\big(\id-t\what{\g}(\mu)\what{\g}(A)\big).\ea
\]
By appealing to Proposition~\ref{prop:basiccaseofdetpq} we see that 
 for $\what{T}=h(\what{\gamma}(\mu A))$ we have
  \[
\tau(t)=\left\{ \ba{ll} 1, &\mbox{ if }t\in (0,1/\what{T}), \\ t^{n-2}, &\mbox{ if
    }t>\what{T}.\ea \right.
\] 
It follows easily from the definitions that  for any square matrix $B$ over $\Z[\what{\pi}]$ we  have $h(\what{\gamma}(B))\leq B$.
In particular we obtain the inequality 
$h(\what{\gamma}(\mu A))\leq h(\mu A)$. By combining this with the inequality (\ref{equ:ha-in-terms-of-entropy}) we see that $\what{T}=h(\what{\gamma}(\mu A))\leq h(f)+\eps$. 
This concludes the proof of the claim and thus also of the theorem.
\end{proof}

%==============================================================================
\section{The $L^2$-Alexander torsion gives a lower bound on the Thurston norm}\label{section:thurston-norm-i}
The goal of this section is to prove Theorem~\ref{thm:lowerbound}.
For the convenience of the reader we recall the statement.\\

\noindent \textbf{Theorem \ref{thm:lowerbound}.}\emph{ Let $(N,\phi,\g)$ be an admissible triple  with $N\ne S^1\times D^2$ and  also $N\ne  S^1\times S^2$  where
  $\g$ is an epimorphism onto a virtually abelian group. Then
  \[
\deg \tautwo(N,\phi,\g) \leq x_N(\phi).
\]}

Theorem~\ref{thm:lowerbound} is an immediate consequence of the following two
propositions. Here note that the first proposition holds without the assumption that the
image of $\g$ is virtually abelian. We also expect the second statement to hold without
any restrictions, but as of now we can not provide a proof.

\begin{proposition} \label{prop:computel2tau}
Let $(N,\phi,\g\co \pi_1(N)\to G)$ be an admissible triple  with $N\ne S^1\times D^2$ and $N\ne S^1\times S^2$. We write 
$H=\ker(\phi\colon G\to \Z)$ and we pick $\mu\in G$ with $\phi(\mu)=1$.
 Then there exist $k,l\in \N$
with $k-l=x_N(\phi)$ and a square matrix $A$ over $\Z[H]$ 
such that 
\[
t\mapsto \max\{1,t\}^{-l}\cdot \detr_{\NN(G)}\left( A+t\mu \bp \id_k&0\\ 0&0\ep\right)   
\]
is a representative of $\tautwo(N,\phi,\g)$. 
\end{proposition}

\begin{proposition} \label{prop:upperboundsdegree}
Let $G$ be a virtually abelian group and let $\phi\co G\to \Z$ be an epimorphism. We write 
$H=\ker(\phi\colon G\to \Z)$. Let $\mu\in G$ with $\phi(\mu)=1$
and let $A$ be a square matrix over $\Z[H]$.
Then
\[
\deg\left(t\mapsto \max\{1,t\}^{-l}\cdot \detr_{\NN(G)}\left( A+t\mu \bp \id_k&0\\ 0&0\ep\right)\right)\leq k-l.
\]
\end{proposition}

%==============================================================================
\subsection{Proof of Proposition~\ref{prop:computel2tau}}
Later on  we will need the following somewhat technical lemma.

\begin{lemma}\label{lem:modifymatrix}
Let $G$ be a group and let $\phi\co G\to \Z$ be an epimorphism. We write 
$H=\ker(\phi\colon G\to \Z)$ and we pick $\mu\in G$ with $\phi(\mu)=1$.
Suppose we are given a square matrix of the form
\[ 
\left(\ba{c|cc|c}  P_1&*&*&P_2\\ \hline 0&\id_{n}& -t^{r}\nu\id_{n}&0\\ *&X &Y&*\\ \hline P_3&*&*&P_4\ea \right)  
\]
where $P_1,P_2,P_3$ and $P_4$ are matrices over $\Z[G]$, 
$\nu\in G$ satisfies $\phi(\nu)=r$, where $X$ and $Y$ are $n\times n$-matrices over $\Z[H]$, and where all other $*$'s indicate matrices over  $\Z[H]$ of an appropriate size. (The vertical and horizontal lines have no mathematical meaning, they are just added to make the matrices more digestible.)
Then there exists an $(r+1)n\times (r+1)n$-matrix $A$ over $\Z[H]$ and further matrices over $\Z[H]$ indicated by $*$ 
such that 
\[ 
\detr_{\NN(G)}\left(\ba{c|cc|c}  P_1&*&*&P_2\\ \hline 0&\id_{n}& -t^{r}\nu\id_{n}&0\\ *&X &Y&*\\ \hline P_3&*&*&P_4\ea\right) = \detr_{\NN(G)}\left(\ba{c|cc|c}  P_1&*&P_2\\ \hline
*& A+t\mu \bp \id_{rn}&0\\ 0&0\ep&*\\ \hline P_3&*&P_4\ea\right)   
\]
for any $t\in \R^+$. 
\end{lemma}

\begin{proof}
Throughout the proof $*$ will always indicate a matrix over $\Z[H]$. 
We consider the following equalities:
\[
\ba{rcl}
 &&\detr_{\NN(G)}\left(\ba{c|cc|c} P_1&*&*&P_2\\\hline 0&\id_{n}& -t^{r}\nu\id_{n}&0 \\  *&X&Y&* \\\hline P_3&*&*&P_4\ea\right) \hspace{4.3cm}\\
 \\[-3.5mm]
 &=&\detr_{\NN(G)}\left(\ba{c|cc|c} P_1&*&*&P_2\\ \hline 0&\id_{n}& -t^{r}\mu^r\id_{n}&0 \\ *&X&Y\nu^{-1}\mu^r&* \\ \hline P_3&*&*&P_4\ea\right)\\
\ea
\]
\[
\ba{rcl}
 % here we really don't do anything, we just insert an identity matrix and then use the block upper triangular matrix formula
&=&\detr_{\NN(G)}\left(\ba{c|ccccc|c} 
P_1&*&0&0&0&*&P_2\\ \hline
0&\id_n&0&\dots&0& -t^r\mu^r\id_n&0 \\ 0&0&\id_n&\ddots&&-t^{r-1}\mu^{r-1}\id_n&0 \\0& 0&\ddots&\ddots &0&\vdots&0\\
0&0&&0 &\id_n &-t\mu \id_n&0\\
*&X&\dots&0&0&Y\nu^{-1}\mu^r&*\\\hline
P_3&*&0&0&0&*&P_4
\ea \right) \\
\\[-3.5mm]
&=&\detr_{\NN(G)}\left(\ba{c|ccccc|c}  
P_1&*&0&0&0&*&P_2\\ \hline
0&\id_n&-t\mu\id_n&\dots&0& 0&0 \\ 0&0&\id_n&\ddots&&0&0 \\ 0&0&\ddots&\ddots &-t\mu\id_n&\vdots&0\\
0&0&&0 &\id_n &-t\mu \id_n&0\\
*&X&\dots&0&0&Y\nu^{-1}\mu^r&*\\ \hline
P_3&*&0&0&0&*&P_4\\
\ea\right). \ea
\]
Here,   we first multiplied the third block column by $\nu^{-1}\mu^r$ on the right. The equality is thus a consequence 
of Proposition~\ref{prop:detl2square} (4). Then we inserted an identity matrix in the center and new entries in the block column on the second to the right, the second equality is thus a consequence of 
Proposition~\ref{prop:detl2square} (2) and (9). Finally for $k=3,\dots,r+1$ we multiplied the $k$-th block row by $- t\mu$ and added it to the previous block row.  Therefore the last equality is a consequence of 
Proposition~\ref{prop:detl2square} (8) and (9).

By swapping the rows appropriately and multiplying them by $-1$ we get the matrix of the
desired form. By Proposition~\ref{prop:detl2square} these procedures do not change the
regular Fuglede--Kadison determinant.
\end{proof}

\begin{lemma}\label{lem:det22}
Let $G$ be a group and let $\phi\co G\to \Z$ be an epimorphism.
Let  $\mu\in G$ with $\phi(\mu)\ne 0$ and let $w\in \ker(\phi)$.
Then for any $t\in \R^+$ we have 
\[
\detr_{\NN(G)}\bp 1 & -t\mu \\ 1 &-w\ep=\max\{1,t\}.
\]
\end{lemma}

\begin{proof}
  We first note that by subtracting the first row from the second row and by
  multiplying the second row by $-w^{-1}$ on the left we turn the given matrix into
  the matrix
  \[
  \bp 1 & -t\mu  \\ 0 & 1-tw^{-1}\mu \ep.
  \]
  Note that $\phi(w^{-1}\mu)\ne 0$, in particular $ w^{-1}\mu$ is an element of $G$ of
  infinite order. The lemma  follows immediately from
  Proposition~\ref{prop:detl2square} and Lemma~\ref{lem:det1minusg}.
\end{proof}

We are now ready to give the proof of Proposition~\ref{prop:computel2tau}.

\begin{proof}[Proof of Proposition~\ref{prop:computel2tau}]
  Let $(N,\phi,\g\co \pi_1(N)\to G)$ be an admissible triple
   with $N\ne S^1\times D^2$ and $N \ne S^1\times S^2$. We write $H=\ker(\phi\colon
  G\to \Z)$ and we pick $\mu\in G$ with $\phi(\mu)=1$.  It follows easily
  from~\cite[Section~1]{Tu02b} that we can find an oriented surface $\S\subset N$ with
  components $\S_1,\dots,\S_l$ and non-zero $r_1,\dots,r_l\in \N$ with the following
  properties: \bn
\item $r_1[\S_1]+\dots+r_l[\S_l]$ is dual to $\phi$,
\item $\sum_{i=1}^l -r_i\chi(\S_i)\leq x_N(\phi)$,
\item  $N\sm \S$ is connected.
\en
For $i=1,\dots,l$  we pick disjoint oriented tubular neighborhoods $\S_i\times [0,1]$ and we identify $\S_i$ with $\S_i\times \{0\}$.
We write $M:=N\sm \cup_{i=1}^l \S_i\times (0,1)$.
We pick once and for all a base point $p$ in $M$ and we denote by $\wti{N}$ the universal cover of $N$.
We write $\pi=\pi_1(N,p)$.
For $i=1,\dots,l$ we also pick a curve $\nu_i$ based at $p$ which intersects $\S_i$ precisely once in a positive direction
and does not intersect any other component of $\S$. Note that $\phi(\nu_i)=r_i$.
By a slight abuse of notation we denote $\g(\nu_i)$ also by $\nu_i$. 
Finally for $i=1,\dots,l$ we write $n_i=-\chi(\S_i)+2$.

Following~\cite[Section~4]{Fr14} we pick an appropriate CW--structure for $N$ and we
pick appropriate lifts of the cells to the universal cover.  The resulting boundary maps
are described in detail~\cite[Section~4]{Fr14}.  In order to keep the notation manageable
we now restrict to the case $l=2$.

It  follows from the discussion in \cite[Section~4]{Fr14}
and the definitions that 
\[ \tau(t):=\tautwo\left(0\to \R[G]^{4}\xrightarrow{B_3}\R[G]^{4+2n_1+2n_2+s}\xrightarrow{B_2}\R[G]^{4+2n_1+2n_2+s}\xrightarrow{B_1}\R[G]^4\to 0\right)\]
is a representative for $\tautwo(N,\phi,\g)$, where $s\in \N$ and where $B_3,B_2,B_1$ are  matrices of the form
\[ \ba{rcl} B_3&=& \hspace{1.33cm} 
\ba{c|ccccccccccc}
&n_1&n_2&1&1&1&1&s+n_1+n_2\\
\hline 
1& * &0&1&-t^{r_1}\nu_1&0&0&0\\ 
1& 0&0&1&-z_1&0&0&*\\
 1& 0&*&0&0&1&-t^{r_2}\nu_2&0\\ 
1&  0&0&0&0&1&-z_2&*\ea\hspace{1.4cm} 
\ea\]
\[ \ba{rcl}
B_2&=&\hspace{-0.1cm} \ba{r|cccccccccccc}
&1&1&n_1&n_1&n_2&n_2&1&1&s\\
\hline
n_1& *&0&\id_{n_1}&-t^{r_1}\nu_1\id_{n_1}&0&0&0&0&0\\   
n_2&0&*&0&0&\id_{n_2}&-t^{r_2}\nu_2\id_{n_2}&0&0&0\\  
1& 0&0&*&0&0&0&0&0&0\\  
1& 0&0&0&*&0&0&0&0&0\\ 
1& 0&0&0&0&*&0&0&0&0\\ 
1& 0&0&0&0&0&*&0&0&0\\ 
s\hspace{-0.1cm}+\hspace{-0.1cm}n_1\hspace{-0.1cm}+\hspace{-0.1cm}n_2& 0&0&*&*&*&*&*&*&*\ea \\
\ea\]
\[\ba{rcl}
B_1&=&\hspace{1.12cm}\ba{c|ccccc}
&1&1&1&1\\
\hline
1& 1&-t^{r_1}\nu_1&0&0\\ 
1&0&0&1&-t^{r_2}\nu_2\\ 
n_1&*&0&0&0\\
n_1&0&*&0&0\\
n_2&0&0&*&0\\
n_2&0&0&0&*\\
1&1&-x_1&0&0 \\
1&0&0&1&-x_2  \\ 
s&*&*&*&*\ea \hspace{5.3cm} \ea \]
with  $x_1,x_2,z_1,z_2\in \g(H)$ and where all the entries of the  matrices marked by $*$ lie in $\Z[\g(H)]$.
Here we use the slightly non-standard, but hopefully useful notation, that the top row indicates the size of the block columns and the left column indicates the size of the block rows. The actual matrix is thus the matrix below the horizontal line and to the right of the vertical line.
(Note that in \cite{Fr14} we view the elements of $\R[G]^n$ as column vectors whereas we now view them as row vectors.)
\medskip

It follows from Lemma~\ref{lem:det22} and Proposition~\ref{prop:detl2square}  that
\[\detr_{\NN(G)} \bp 1& -t^{r_1}\nu_1&0&0 \\1&-z_1&0&0 \\ 0&0&1& -t^{r_2}\nu_2 \\ 0&0&1&-z_2\ep
=\max\{1,t^{r_1+r_2}\}. \]
If we write $*_{i\times j}$ for an $i \times j$-matrix, then it follows from  Lemma~\ref{lem:torsion3complex}  that $\tau(t)$ equals
\[
\\
\max\{1,t^{r_1+r_2}\}^{-1}\cdot 
 \detr_{\NN(G)}\hspace{-0.2cm}\bp \id_{n_1}\hspace{-0.1cm}&\hspace{-0.11cm} -t^{r_1}\nu_1\id_{n_1}\hspace{-0.1cm}&\hspace{-0.11cm}0\hspace{-0.1cm}&\hspace{-0.11cm}0\hspace{-0.1cm}&\hspace{-0.11cm}0\\ 0\hspace{-0.1cm}&\hspace{-0.11cm}0\hspace{-0.1cm}&\hspace{-0.11cm}\id_{n_2}\hspace{-0.1cm}&\hspace{-0.11cm}  -t^{r_2}\nu_2\id_{n_2}\hspace{-0.1cm}&\hspace{-0.11cm}0\\ *\hspace{-0.1cm}&\hspace{-0.11cm}*\hspace{-0.1cm}&\hspace{-0.11cm}*\hspace{-0.1cm}&\hspace{-0.11cm}*\hspace{-0.1cm}&\hspace{-0.11cm}*_{(n_1+n_2+s)\times s}\ep\hspace{-0.1cm}
 \cdot \max\{1,t^{r_1+r_2}\}^{-1}\]
which we can rewrite as
\[
\max\{1,t^{r_1+r_2}\}^{-1}\cdot 
 \detr_{\NN(G)}\hspace{-0.2cm}\bp \id_{n_1}\hspace{-0.1cm}&\hspace{-0.11cm} -t^{r_1}\nu_1\id_{n_1}\hspace{-0.1cm}&\hspace{-0.11cm}0\hspace{-0.1cm}&\hspace{-0.11cm}0\hspace{-0.1cm}&\hspace{-0.11cm}0\\
 *_{n_1\times n_1} \hspace{-0.1cm}&\hspace{-0.11cm} *_{n_1\times n_1} \hspace{-0.1cm}&\hspace{-0.11cm}*\hspace{-0.1cm}&\hspace{-0.11cm}*\hspace{-0.1cm}&\hspace{-0.11cm}*\\
  0\hspace{-0.1cm}&\hspace{-0.11cm}0\hspace{-0.1cm}&\hspace{-0.11cm}\id_{n_2}\hspace{-0.1cm}&\hspace{-0.11cm}  -t^{r_2}\nu_2\id_{n_2}\hspace{-0.1cm}&\hspace{-0.11cm}0\\
  *\hspace{-0.1cm}&\hspace{-0.11cm}*\hspace{-0.1cm}&\hspace{-0.11cm} *_{n_2\times n_2} \hspace{-0.1cm}&\hspace{-0.11cm} *_{n_2\times n_2} \hspace{-0.1cm}&\hspace{-0.11cm}*\\
   *\hspace{-0.1cm}&\hspace{-0.11cm}*\hspace{-0.1cm}&\hspace{-0.11cm}*\hspace{-0.1cm}&\hspace{-0.11cm}*\hspace{-0.1cm}&\hspace{-0.11cm}*_{s\times s}\ep
\hspace{-0.1cm} \cdot \max\{1,t^{r_1+r_2}\}^{-1}.
\]
Now we set $l=2r_1+2r_2$ and $k=r_1n_1+r_2n_2$.  It is then straightforward to see that if
we apply Lemma~\ref{lem:modifymatrix} twice then we can turn the above matrix
into a matrix of the desired form. We leave the elementary details to the reader.
\end{proof}

%==============================================================================
\subsection{Proof of Proposition~\ref{prop:upperboundsdegree}}
It is clear that the following proposition, together with elementary properties of the
degree function, implies Proposition~\ref{prop:upperboundsdegree}.

\begin{proposition} \label{prop:upperboundsdegreeb} Let $G$ be a virtually abelian
  group. Let $m\geq k$ be natural numbers. Let $A$ be an $m\times m$-matrix over $\Z[G]$
  and let $B$ be a $k\times k$-matrix over $\Z[G]$. Then
  \[\deg\left(t\mapsto \detr_{\NN(G)}\left( A+t\bp B&0\\0&0\ep\right)\right)\leq k.\]
\end{proposition}

\begin{proof}[Proof of Proposition~\ref{prop:upperboundsdegreeb}]
For $t\in \R^+$ we define
\[
f(t):=\detr_{\NN(G)}\left( A+t\bp B&0\\0&0\ep\right).
\]
It suffices to prove the following claim.

\begin{claim}
\[
\deg_0(f(t))\geq 0\mbox{ and }\deg_\infty(f(t))\leq k.
\]
\end{claim}

We start out with $ \deg_0(f(t))$. If $f(t)=0$ for arbitrarily small $t$, then there is
nothing to prove. Now we suppose that this is not the case.  It follows from
Corollary~\ref{cor:detr-continuous} that
\[
\ba{rcl} \underset{t\to 0}{\lim}\,  f(t)&=&\underset{t\to 0}{\lim}\, \detr_{\NN(G)}\left( A+t\bp B&0\\0&0\ep\right)\\
\\[-3mm]
& =&
 \detr_{\NN(G)}\underset{t\to 0}{\lim}\,\left( A+t \bp B&0\\0&0\ep\right)=
 \detr_{\NN(G)}\left( A\right)\in [0,\infty).\ea 
\]
In particular we see that  $\ln(f(t))$ is bounded from the above for sufficiently small $t$. It  follows that 
\[
\deg_0(f(t))= \lim_{t\to 0}\frac{\ln f(t)}{\ln t}\geq 0.\]
Now we turn to $\deg_\infty(f(t))$. We write
\[
A=\bp A_1&A_2\\ A_3&A_4\ep
\]
where $A_1$ is a $k\times k$-matrix.
It then  follows from Proposition~\ref{prop:detl2square} and Corollary~\ref{cor:detr-continuous} that 
\[
\ba{rcl}
\underset{t\to \infty}{\lim}\,\frac{1}{t^k}f(t)&=&
\underset{t\to \infty}{\lim}\,\frac{1}{t^k}\detr_{\NN(G)}\bp A_1+t B&A_2\\A_3&A_4\ep\\ \\[-3mm]
&=&\underset{t\to \infty}{\lim}\detr_{\NN(G)}\bp t^{-1} A_1+ B&t^{-1}A_2\\A_3&A_4\ep\\ \\[-3mm]
&=&\detr_{\NN(G)}\underset{t\to \infty}{\lim}\, \bp t^{-1} A_1+ B&t^{-1}A_2\\A_3&A_4\ep\\ \\[-3mm]
&=&\detr_{\NN(G)} \bp  B&0\\A_3&A_4\ep\in [0,\infty).\ea 
\]
Thus it follows that $\ln\left( \frac{1}{t^k}f(t)\right)$ is bounded from the above for sufficiently large $t$.
Now we see that 
\[
\deg_\infty(f(t))=\lim_{t\to \infty}\frac{\ln {f(t)}}{\ln t}=
\lim_{t\to \infty}\frac{\ln \left(t^k\frac{1}{t^k}f(t)\right)}{\ln t}
=k+\lim_{t\to \infty}\frac{\ln \left(\frac{1}{t^k}f(t)\right)}{\ln t}\leq k.
\]
This concludes the proof of the claim and thus of the proposition. 
\end{proof}

%==============================================================================
\section{The $L^2$-Alexander torsion detects  the Thurston norm}
\label{section:thurston-norm-ii}
In Section~\ref{thm:graph} we had already seen that `most' $L^2$-Alexander torsions detect
the Thurston norm of a graph manifold.  In this section we will show that also for all
other prime 3-manifolds there exists an $L^2$-Alexander torsion which detects the Thurston
norm.  More precisely, the goal of this section is to prove the following theorem from the introduction.\\

\noindent \textbf{Theorem \ref{thm:detectsnorm}.}\emph{
  Let $N$ be a prime 3-manifold  that is not a closed graph
  manifold.  Then there exists an epimorphism $\g\co \pi_1(N)\to G$ onto a virtually
  abelian group such that the projection map $\pi_1(N)\to H_1(N;\Z)/\mbox{torsion}$
  factors through $\g$ and such that for any $\phi\in H^1(N;\R)$ the function
  $\tautwo(N,\phi,\g)$ is monomial in the limit with
  \[
  \deg \tautwo(N,\phi,\g)=x_N(\phi).
  \]}

%==============================================================================
\subsection{The Virtual Fibering Theorem}\label{section:virtfib}

Before we state the Virtual Fibering Theorem of Agol~\cite{Ag08} we
need to recall a few definitions.  First of all, given a $3$-manifold
$N$ we say that a class $\phi\in H^1(N;\R)$ is \emph{quasi-fibered} if
$\phi$ is the limit of fibered classes in $H^1(N;\Q)$.  We will also
use the notion of a group $\pi$ being \emph{residually finite
  rationally solvable \textup{(}RFRS\textup{)}}. In fact, as we will soon seen, for the purpose of this paper  one can treat
this notion as a black box. Thus we provide the definition only for
completeness' sake.  A group is RFRS if there exists a filtration of
$\pi$ by subgroups $\pi=\pi_0\supseteq \pi_1 \supseteq \pi_2\cdots $
such that \bn
\item $\bigcap_i \pi_i=\{1\}$,
\item for any $i$ the group  $\pi_i$ is a normal, finite-index subgroup of  $\pi$, 
\item for any $i$ the map $\pi_i\to \pi_i/\pi_{i+1}$ factors through $\pi_i\to H_1(\pi_i;\Z)/\mbox{torsion}$.
\en

The following is a straightforward consequence of the virtual fibering theorem of  
Agol~\cite[Theorem~5.1]{Ag08} (see also~\cite[Theorem~5.1]{FK14} and~\cite[Corollary~5.2]{FV12}).

\begin{theorem}\label{thm:quasifib}
Let $N$ be a prime 3-manifold.
Suppose that  $\pi_1(N)$ is virtually {RFRS}.
Then there exists a finite regular cover $p\colon \widehat{N}\to N$  such that for every
class $\phi\in H^1(N;\R)$ the class $p^*\phi\in H^1(\widehat{N};\R)$ is quasi-fibered.
\end{theorem}

The following theorem was proved by Agol~\cite{Ag13} and Wise~\cite{Wi12a,Wi12b} in the
hyperbolic case.  It was proved by Liu~\cite{Liu13} and Przytycki--Wise~\cite{PW14} for
graph manifolds with boundary and it was proved by Przytycki--Wise~\cite{PW12} for
manifolds with a non-trivial JSJ decomposition and at least one hyperbolic piece in the
JSJ decomposition.

\begin{theorem}\label{thm:virtrfrs}
If  $N$ is a prime $3$-manifold that is not a closed graph manifold, then $\pi_1(N)$ is virtually RFRS.
\end{theorem}

%==============================================================================
\subsection{Continuity of degrees}

Given a group $G$, a homomorphism $\phi\co G\to \R$ and   $t\in \R^+$ we  consider the ring homomorphism
\[
\ba{rrcl} \kappa(\phi,t)\co &\R[G] &\to& \R[G] \\
&\sum\limits_{i=1}^n a_ig_i&\mapsto & \sum\limits_{i=1}^n a_it^{\phi(g_i)}g_i.\ea 
\]
As usual, given a matrix $A$ over $\R[G]$ we define $\kappa(\phi,t)(A)$ by applying $\kappa(\phi,t)$ to each entry of $A$.
% Given a square matrix $A$ over $\R[\pi]$
% we write $\det(A,\phi,t):=\detr_{\NN(G)}(\kappa(\phi,t)(A))$. 
%Note that this defines a function
%\[
%\ba{rcl} \det(A,\phi,)\co \R^+&\to& [0,\infty) \\
%t&\mapsto &\det(A,\phi,t).\ea 
%\]

Recall that in Section~\ref{section:functiondegree} we associated to many functions $f\co
\R^+\to [0,\infty)$ a degree $\deg(f)$ with values in $\R\cup \{\pm \infty\}$. Now we
endow $\R\cup \{\pm \infty\}$ with the usual topology, i.e.,  the topology on $\R$ with a
`point at infinity on the left' and a `point at infinity on the right'.

We have the following proposition.

\begin{proposition}\label{prop:degcontinuous}
Let $G$  be a virtually abelian group and  let $A$ be a square matrix over $\Z[G]$
such that $\detr_{\NN(G)}(A)\ne 0$. Then the map
\[ 
\ba{rcl} \hom(G,\R)&\to & \R\cup \{\pm \infty\} \\
\phi&\mapsto &\deg\left(\ba{rcl} \R^+&\to &[0,\infty)\\ t&\mapsto& \detr_{\NN(G)}\big(\kappa(\phi,t)(A)\big)\ea \right)\ea 
\]
takes values in $[0,\infty)$ and it is a (possibly degenerate) norm.
\end{proposition}

Before we continue, recall that given a free abelian group $F$ and $p\in \R[F]$ we denote by $m(p)$ the Mahler measure.
In the proof of Proposition~\ref{prop:degcontinuous} we will need the following lemma.

\begin{lemma}\label{lem:alexnorm}
Let $F$ be a free abelian group and let $p\in \R[F]$ be non-zero.
We write $p=\sum_{f\in F} a_f\cdot f$. Then for any $\phi\in \hom(F,\R)$ we have
\[
\deg\big(t\mapsto m(\kappa(\phi,t)(p))\big)=\max\{ \phi(f)-\phi(g)\,|\,f,g\in F\mbox{ with }a_f\ne 0\mbox{ and }a_g\ne 0\}.
\] 
\end{lemma}

\begin{proof}
Let  $\phi\in \hom(F,\R)$. We denote by
\[ S:=\{ f\in F\,|\, a_f\ne 0\}\]
the support of $p=\sum_{f\in F} a_f\cdot f$. 
We write 
\[
 d=\min\{ \phi(s)\,|\,s\in S\}\mbox{ and }
D=\max\{ \phi(s)\,|\,s\in S\}.\]
Now we sort the summands of $p$ according to their $\phi$-values. More precisely, since $\phi$ takes only finitely many values on $S$ we can find  $p_1,\dots,p_r\in \R[\ker(\phi)]$ and $g_1,\dots,g_r\in F$ with 
\[
d=\phi(g_1)<\phi(g_2)<\dots <\phi(g_r)=D
\]
such that $p=p_1g_1+\dots+p_rg_r$. Note that $p_1\ne 0$ and $p_r\ne 0$ by definition of $d$ and $D$. By the continuity
of the Mahler measure, see Corollary~\ref{cor:detr-continuous} and~\cite[p.~127]{Bo98}, we
have
\[
\ba{rcl} 
\underset{t\to \infty}{\lim}\, \smfrac{m(\kappa(\phi,t)(p))}{t^D}&=&
\underset{t\to \infty}{\lim}\, m\left(p_1g_1\tmfrac{t^{\phi(g_1)}}{t^D}+\dots+p_rg_r\tmfrac{t^{\phi(g_r)}}{t^D}\right)\\[2mm]
&=&
 m\left(\underset{t\to \infty}{\lim}\,\left(p_1g_1\tmfrac{t^{\phi(g_1)}}{t^D}+\dots+p_rg_r\tmfrac{t^{\phi(g_r)}}{t^D}\right)\right)=m(p_rg_r)\ne 0.\ea
\]
It thus follows that $\deg_\infty(t\mapsto m(\kappa(\phi,t)(p)))=D$.  Basically the same
argument also shows that $\deg_0(t\mapsto m(\kappa(\phi,t)(p)))=d$. Putting these two
equalities together gives the desired result.
\end{proof}

We can now give the proof of Proposition~\ref{prop:degcontinuous}.

\begin{proof}[Proof of Proposition~\ref{prop:degcontinuous}]
Let $G$ be a virtually abelian group. There exists  a 
finite index subgroup $F$ that is torsion-free abelian.  We pick representatives
  $g_1,\dots,g_d$ for $G/F$. Given a matrix $B$ over $\R[G]$ we define the matrix
  $\iota^{F}_G(B)$ over $\R[F]$ as in Section~\ref{section:propfk}, using this ordered set
  of representatives.  It follows easily from the definitions that for any $\phi\in
  \hom(G,\R)$ and any $t\in \R^+$ we have \be \label{equ:iotarho}
  \iota^{F}_G(\kappa(\phi,t)(A))=\kappa(\phi|_F,t)\left(\iota^{F}_G(A)\right).\ee Now we
  denote by $p\in \Z[F]$ the determinant of $\iota^{F}_G(A)$.  It follows from
  (\ref{equ:iotarho}), Proposition~\ref{prop:detl2square} and Lemma~\ref{lem:detl2mahler}
  that \be\label{equ:detrho}
  \detr_{\NN(G)}\big(\kappa(\phi,t)(A)\big)=m\big(\kappa(\phi,t)(p)\big)^{\frac{1}{[G:F]}}
  \;\mbox{for any $\phi\in \hom(G,\R)$ and $t\in \R^+$.}\ee If we apply (\ref{equ:detrho}) to
  $t=1$, then we see that our assumption that $\detr_{\NN(G)}(A)\ne 0$ implies in
  particular that $p\ne 0$.  Furthermore, by the combination of (\ref{equ:detrho}) and
  Lemma~\ref{lem:degreefunction} (7) it suffices to show that the map
  \[
  \ba{rcl} \hom(F,\R)&\to & \R\cup \{\pm \infty\}\\
  \psi&\mapsto &\deg\big(t\mapsto m(\kappa(\psi,t)(p))\big)\ea 
  \] 
  takes values in $[0,\infty)$ and that it is a (possibly degenerate) norm.  But since $p\ne 0$ this is an
  immediate consequence of Lemma~\ref{lem:alexnorm}.
\end{proof}

%==============================================================================
\subsection{The proof of Theorem~\ref{thm:detectsnorm}}

\begin{proof}
  Let $N$ be a prime 3-manifold which is not a closed graph manifold.  It follows from
  Theorems~\ref{thm:quasifib} and~\ref{thm:virtrfrs} that there exists a finite regular
  cover $p\colon \widehat{N}\to N$ such that given any $\phi\in H^1(N;\R)$ the pull-back
  $p^*\phi\in H^1(\what{N};\R)$ is quasi-fibered.

  Now we denote by $\what{\g}\colon \pi_1(\what{N})\to H:=H_1(\what{N};\Z)/\mbox{torsion}$
  the canonical epimorphism.  By Theorem~\ref{thm:l2fibered} we have
  \be\label{equ:taux}\deg \tautwo(\what{N},\what{\g},{\psi})=x_{\what{N}}({\psi})\;
  \mbox{for any fibered }\psi\in H^1(\what{N};\Q).\ee It follows from
  Proposition~\ref{prop:degcontinuous} and from the fact that $x_{\what{N}}$ is a norm
  that both sides of (\ref{equ:taux}) are continuous in $\psi$. It thus follows that we
  also have
  \[
  \deg \tautwo(\what{N},\what{\g},{\psi})=x_{\what{N}}({\psi})\mbox{ for any
    quasi-fibered }\psi\in H^1(\what{N};\R).
   \] 
   In particular the equality holds for any   $p^*\phi$ with $\phi\in H^1(N;\R)$.

Now we consider the projection homomorphism 
\[
\g\co \pi_1(N)\to G:=\pi_1(N)/\ker\{\what{\gamma}\co\pi_1(\what{N})\to H\}.
\]
(Note that $\ker\{\what{\gamma}\co\pi_1(\what{N})\to H\}$ is characteristic in $\pi_1(\what{N})$ hence it is  normal in $\pi_1(N)$.)
It follows from the above, from Lemma~\ref{lem:finitecover} and the multiplicativity of
the Thurston norm under finite covers (see Gabai \cite[Corollary~6.13]{Ga83}), that for any
$\phi\in H^1(N;\R)$ we have
\[
 \deg \tautwo({N},{\g},{\phi})=\frac{1}{[\what{N}:N]}\deg \tautwo(\what{N},
\what{\g},p^*\phi)=\frac{1}{[\what{N}:N]}x_{\what{N}}(p^*\phi)=x_N(\phi).
\]
\end{proof}

%\version{31.10.2014}


\begin{thebibliography}{10}

\bibitem[Ag08]{Ag08}
I. Agol, {\em Criteria for virtual fibering}, J. Topol. 1 (2008), no. 2, 269--284.
\bibitem[Ag13]{Ag13}
I. Agol, {\em The virtual Haken conjecture}, with an appendix by I. Agol, D. Groves and J. Manning, 
	Documenta Math. 18 (2013), 1045--1087.
\bibitem[Ah78]{Ah78}
L. Ahlfors, {\em  Complex analysis: An introduction to the theory of analytic functions of one complex variable}, Third edition. International Series in Pure and Applied Mathematics. McGraw-Hill Book Co., New York, 1978.
%\bibitem[AFW12]{AFW12}
%M. Aschenbrenner, S. Friedl and H. Wilton, {\em $3$-manifold groups}, Preprint (2012)
\bibitem[BA13a]{BA13a}
F. Ben Aribi, {\em The $L^2$-Alexander invariant detects the unknot}, C. R. Math. Acad. Sci. Paris 351 (2013), 215--219.
\bibitem[BA13b]{BA13b}
F. Ben Aribi, {\em The $L^2$-Alexander invariant detects the unknot}, Preprint (2013), to be published by the Annali della Scuola Normale Superiore die Pisa.
\bibitem[Bo98]{Bo98}
D. Boyd, {\em Uniform approximation to Mahler''s measure in several variables}, Canad. Math. Bull. 41 (1998), no. 1, 125--128.
\bibitem[CFM97]{CFM97}
A. Carey, M. Farber and V. Mathai, {\em Determinant lines, von Neumann algebras and $L^2$ torsion}, J. Reine Angew. Math. 484 (1997), 153--181.
\bibitem[Ch74]{Ch74}
T. A. Chapman, {\em Topological invariance of Whitehead torsion},
Amer. J. Math.  96, No. 3 (1974), 488--497.
%\bibitem[CO09]{CO09}
%J. Cha and K. Orr, {\em $L^2$-signatures, homology localization and amenable groups}, Preprint (2009)
\bibitem[Cl99]{Cl99}
B. Clair, {\em Residual amenability and the approximation of $L^2$-invariants},
Michigan Math. J. 46 (1999), 331--346.
\bibitem[Co04]{Co04}
T. Cochran,  {\em Noncommutative knot theory}, Algebr. Geom.
Topol. {4} (2004), 347--398.
%\bibitem[COT03]{COT03}
%T. Cochran, K. Orr and  P. Teichner, {\em Knot concordance, Whitney towers and $L\sp 2$-signatures},
%Ann. of Math. (2)  157,  no. 2: 433--519 (2003)
\bibitem[CF63]{CF63}
R. H. Crowell and R. H. Fox, {\em Introduction to knot theory}, Ginn and Co., Boston, 1963. 
%\bibitem[Di81]{Di81}
%J. Dixmier, {\em Von Neumann Algebra}, North-Holland Publishing Co., Amsterdam,
%1981.
\bibitem[DFL14]{DFL14}
J. Dubois, S. Friedl and W. L\"uck, {\em $L^2$-Alexander torsions are symmetric}, preprint (2014), to be published by Alg. Geom. Top.
\bibitem[DFL15]{DFL15}
J. Dubois, S. Friedl and W. L\"uck, {\em Three flavors of twisted knot invariants}, Introduction to Modern Mathematics, Advanced Lectures in Mathematics  33 (2015), 143--170.
\bibitem[DW10]{DW10}
J. Dubois and C. Wegner, {\em $L^2$-Alexander invariant for torus knots}, C. R. Math. Acad. Sci. Paris 348 (2010), no. 21-22, 1185-1189.
\bibitem[DW15]{DW15}
J. Dubois and C. Wegner, {\em $L^2$-Alexander invariant for knots},  Commun. Contemp. Math. 17 (2015), no. 1, 1450010, 29 pp.
\bibitem[EN85]{EN85}
D.  Eisenbud and  W. Neumann, {\em  Three-dimensional link theory and invariants of plane curve
singularities}, Annals of Mathematics Studies, 110. Princeton University Press, Princeton, NJ,
1985.
\bibitem[ES05]{ES05}
G. Elek and E. Szab\'o, {\em
Hyperlinearity, essentially free actions and $L^2$-invariants},
Math. Ann. 332 (2005), 421--441.
\bibitem[ES06]{ES06}
G. Elek and E. Szab\'o, {\em On sofic groups}, 
J. Group Theory 9 (2006), 161--171.
\bibitem[FLP79]{FLP79}
A. Fathi, F. Laudenbach and V. Po\'enaru, {\em Travaux de Thurston sur les surfaces}, Ast\'erisque,
66-67, Soc. Math. France, Paris, 1979.
\bibitem[Fo53]{Fo53}
R. H. Fox, {\em Free differential calculus I, Derivation in the free group ring}, Ann. Math. 57 (1953), 547--560.
%\bibitem[Fr07]{Fr07}
%S. Friedl, {\em Reidemeister torsion, the Thurston norm and Harvey's invariants},
% Pac. J. Math. 230 (2007), 271--296.
\bibitem[Fr14]{Fr14}
S. Friedl,  {\em Twisted Reidemeister torsion, the Thurston norm and fibered manifolds}, Geom. Dedicata 172,  (2014),  135--145.
%\bibitem[FH07]{FH07}
%S. Friedl and S. Harvey, {\em
%Non-commutative Multivariable Reidemeister Torsion and the Thurston Norm},
%Alg. Geom. Top. 7 (2007), 755--777.
\bibitem[FJR11]{FJR11}
S. Friedl, A. Juh\'asz and J. Rasmussen,  {\em The decategorification of sutured Floer homology}, J. Top.  4 (2011), 431-478.
%\bibitem[FKm06]{FKm06}
% S. Friedl and  T. Kim, \emph{Thurston norm, fibered manifolds and twisted Alexander
%polynomials}, Topology  45 (2006), 929--953.
%\bibitem[FKm08a]{FKm08a}
%S. Friedl and T. Kim, {\em  Twisted Alexander norms give lower bounds on the Thurston norm},
%Trans. Amer. Math. Soc. 360 (2008), 4597--4618
%\bibitem[FKm08b]{FKm08b}
%S. Friedl and T. Kim, {\em The parity of the Cochran-Harvey invariants of
%3-manifolds},
%    Trans. Amer. Math. Soc. 360 (2008), 2909--2922.
\bibitem[FKK12]{FKK12}
S. Friedl, T. Kim and T. Kitayama,
{\em Poincar\'e duality and degrees of twisted Alexander polynomials}, Ind. Univ. Math. J. 61 (2012), 147--192. 
\bibitem[FK14]{FK14}
S. Friedl and T. Kitayama, {\em The virtual fibering theorem for $3$-manifolds}, L'Enseignement Math\'ematique  60 (2014), no. 1, 79--107. 
\bibitem[FLM09]{FLM09}
S. Friedl, C. Leidy and L. Maxim, {\em $L^2$-Betti numbers of plane algebraic curves},
Michigan Math. Journal, 58 (2009), no. 2, 291--301.
\bibitem[FL15]{FL15}
S. Friedl and W. L\"uck, {\em The {$L^2$}-torsion function and the {T}hurston norm of $3$-manifolds}, Preprint (2015).
\bibitem[FV10]{FV10}
S. Friedl and  S. Vidussi,
 {\em A survey of twisted Alexander polynomials}, The Mathematics of Knots: Theory and Application (Contributions in Mathematical and Computational Sciences), editors: Markus Banagl and Denis Vogel (2010),  45--94.
 \bibitem[FV12]{FV12}
 S. Friedl and  S. Vidussi,
  {\em  The Thurston norm and twisted Alexander polynomials}, preprint (2012), to appear in the Journal f\"ur reine und angewandte Mathematik.
%\bibitem[Ge83]{Ge83}
%S. Gersten, {\em Conservative groups, indicability, and a conjecture of Howie},
%J. Pure Appl. Algebra 29 (1983), no. 1, 59--74.
\bibitem[Ga83]{Ga83}
D. Gabai, {\em Foliations and the topology of $3$-manifolds}, J. Differential Geometry 18 (1983), no. 3,
445--503.
%\bibitem[GKM05]{GKM05} H. Goda, T. Kitano and  T. Morifuji, {\em Reidemeister Torsion, Twisted Alexander Polynomial and Fibred Knots}, Comment. Math. Helv.  80 (2005), 51--61.
\bibitem[Go83]{Go83}
C. McA. Gordon, {\em Dehn surgery and satellite knots}, Trans. Amer. Math. Soc.
275 (1983), 687--708.
\bibitem[Gr99]{Gr99}
M. Gromov, {\em Endomorphisms of symbolic algebraic varieties}, J. Eur Math. Soc.  1 (1999), 109--197.
\bibitem[Ha05]{Ha05}
S. Harvey, {\em Higher--order polynomial invariants of 3--manifolds giving lower bounds for the
Thurston norm}, Topology {44} (2005), 895--945.
%\bibitem[Ha06]{Ha06}
%S. Harvey, {\em Monotonicity of degrees of generalized Alexander polynomials
%of groups and 3-manifolds}, Math. Proc. Camb. Phil. Soc. 140 (2006),  431--450.
%\bibitem[Ha08]{Ha08}
%S. Harvey, {\em Homology Cobordism Invariants and the Cochran-Orr-Teichner Filtration of the
%Link Concordance Group}, Geom. Topol.  12 (2008), 387--430.
\bibitem[Hem87]{Hem87} J. Hempel, {\em Residual
finiteness for $3$-manifolds}, Combinatorial group theory and topology (Alta, Utah, 1984),
379--396, Ann.\ of Math. Stud., 111, Princeton Univ. Press, Princeton, NJ (1987)
\bibitem[Her15]{Her15}
G. Herrmann, {\em The $L^2$-Alexander torsion of Seifert fibered spaces}, Masters thesis (2015), University of Regensburg.
\bibitem[HSW10]{HSW10}
J. Hillman, D. Silver and S. Williams,
{\em On reciprocality of twisted Alexander invariants}, Alg. Geom. Top. 10 (2010), 2017--2026.
\bibitem[Ho82]{Ho82}
J. Howie, {\em On locally indicable groups}, Math. Z. 180 (1982), 445--461.
\bibitem[KM14]{KM14}
S. Kojima and G. McShane,
{\em Normalized entropy versus volume for pseudo-Ansovs}, Preprint (2014)

%\bibitem[HS83]{HS83}
%J. Howie and H. Schneebeli, {\em Homological and topological properties of locally indicable groups}, Manuscripta Math. 44 (1983), no. 1-3, 71–93.
%\bibitem[Ki97]{Ki97} R. Kirby, {\em  Problems in low-dimensional topology}, Edited by Rob Kirby. AMS/IP Stud. Adv. Math., 2.2,  Geometric topology (Athens, GA, 1993),  35--473, Amer. Math. Soc., Providence, RI, 1997. 57-02
%\bibitem[Lei06]{Lei06}
%C. Leidy, {\em   	
%Higher-order linking forms for knots},
%Comm. Math. Helv.  81 (2006), 755-781
%\bibitem[LM06]{LM06}
%C. Leidy and L. Maxim, {\em Higher-order Alexander invariants of plane algebraic curves},
% IMRN (2006), Article ID 12976, 23 pages.
%\bibitem[LM08]{LM08}
%C. Leidy and L. Maxim, {\em
%Obstructions on fundamental groups of plane curve complements},
%Real and Complex Singularities, Contemporary Mathematics, 459 (2008), 117--130.
%
\bibitem[LZ06a]{LZ06a}
W. Li and  W. Zhang, {\em An $L^2$-Alexander invariant for knots},  Commun. Contemp. Math. 8
(2006), no. 2, 167--187.
\bibitem[LZ06b]{LZ06b}
W. Li and W. Zhang, {\em An $L^2$-Alexander-Conway invariant for knots and the volume conjecture},
Differential geometry and physics, 303--–312, Nankai Tracts Math., 10, World Sci. Publ., Hackensack, NJ, 2006.
\bibitem[LZ08]{LZ08}
W. Li and W. Zhang, {\em Twisted $L^2$-Alexander-Conway invariants for knots}, Topology and physics, 236--–259, Nankai Tracts Math., 12, World Sci. Publ., Hackensack, NJ, 2008.
\bibitem[Li01]{Li01} X. S. Lin, {\em Representations of knot groups and twisted
Alexander polynomials}, Acta Math. Sin. (Engl. Ser.)  17,  no. 3 (2001), 361--380.
%\bibitem[Lil93]{Lil93}
%P. Linnell, {\em Division rings and group von Neumann algebras},
%Forum Math 5 (1993), no. 6, 561--576.
\bibitem[Liu13]{Liu13}
Y. Liu, {\em Virtual cubulation of
nonpositively curved graph manifolds}, J. of Topology 6 (2013), 793--822.
\bibitem[Liu15]{Liu15}
Y. Liu, {\em 
Degree of $L^2$-Alexander torsion for 3-manifolds}, Preprint (2015).
%\bibitem[LL95]{LL95}
%J.  Lott and  W. L\"uck, {\em $L\sp 2$-topological invariants of $3$-manifolds},  Invent. Math.
%120 (1995), no. 1, 15--60.
\bibitem[L\"u94]{Lu94}
W. L\"uck, {\em Approximating $L^2$-invariants by their finite-dimensional analogues},  Geom.
Funct. Anal. 4 (1994), 455--481.
\bibitem[L\"u02]{Lu02} W. L\"uck, {\em
 $L\sp 2$-invariants: theory and applications to geometry and $K$-theory}, Ergebnisse
der Mathematik und ihrer Grenzgebiete. 3. Folge. A Series of Modern Surveys in Mathematics, 44.
Springer-Verlag, Berlin, 2002.
\bibitem[LS99]{LS99}
W. L\"uck and  T. Schick, {\em $L\sp 2$-torsion of hyperbolic manifolds of finite volume}, Geom.
Funct. Anal.  9  (1999),  no. 3, 518--567.
%\bibitem[Ma13]{Ma13}
%L. Maxim, {\em $L^2$-Betti numbers of hypersurface complements}, 
%Int. Math. Res. Not.  (2013)
\bibitem[Mi66]{Mi66}
J. Milnor, {\em Whitehead torsion}, Bull. Amer. Math. Soc. 72 (1966), 358--426.
 
%\bibitem[Mc02]{Mc02} C. T. McMullen, {\em The Alexander polynomial of a 3--manifold and the Thurston
%norm on cohomology}, Ann. Sci. Ecole Norm. Sup. (4) 35 (2002), 153--171.
\bibitem[Ne99]{Ne99}
W. Neu\-mann, {\em  
Notes on geometry and 3-manifolds}, Bolyai Soc. Math. Stud., 8, Low dimensional topology (Eger, 1996/Budapest, 1998), 191--267, J\'anos Bolyai Math. Soc., Budapest, 1999. 
\bibitem[Pa77]{Pa77}
D. Passman, {\em The algebraic structure of group rings}, 
 John Wiley \& Sons. XIV (1977).
\bibitem[PW12]{PW12}
 P. Przytycki and D. Wise, {\em Mixed $3$-manifolds are virtually special}, Preprint (2012).
\bibitem[PW14]{PW14}
 P. Przytycki and D. Wise, {\em Graph manifolds with boundary are virtually special},  Journal of Topology 7 (2014), 419-435. 
\bibitem[Ra12]{Ra12}
J. Raimbault, {\em Exponential growth of torsion in abelian coverings},
Algebr. Geom. Topol. 12 (2012), 1331--1372.
\bibitem[Ro90]{Ro90}
D. Rolfsen, {\em Knots and Links}, Mathematics Lecture Series, vol. 7, Publish or Perish, Inc., Houston, TX, 1990.
\bibitem[Sc01]{Sc01}
T. Schick, {\em $L^2$-determinant class and approximation of $L^2$--Betti numbers},
Trans. Amer. Math. Soc. 353 (2001), 3247--3265.
%\bibitem[Sct95]{Sct95}
%K. Schmidt, {\em Dynamical Systems of Algebraic Origin}, Birkh\"auser Verlag, Basel, 1995.
%\bibitem[SW04]{SW04}
%D. Silver and S. Williams, {\em
% Mahler measure of Alexander polynomials}, J. London Math. Soc. (2) 69 (2004), 767--782.
%\bibitem[SW10]{SW10}
%D. Silver and S. Williams, {\em Twisting Alexander Invariants with Periodic Representations}, preprint (2010)
%\bibitem[St74]{St74}
%R. Strebel, {\em Homological methods applied to the derived series of groups}, Comment. Math. Helv. 49 (1974), 302--332.
\bibitem[Th86]{Th86} W. P. Thurston, {\em A norm for the homology of 3--manifolds}, Mem.
Amer. Math. Soc. 339: 99--130 (1986)
\bibitem[Tu86]{Tu86} V. Turaev, {\em
Reidemeister torsion in knot theory},
 Russian Math. Surveys 41 (1986), no. 1, 119--182.
\bibitem[Tu01]{Tu01} V. Turaev, {\em Introduction to combinatorial torsions}, Birkh\"auser, Basel, (2001)
\bibitem[Tu02a]{Tu02a}
V. Turaev, {\em Torsions of 3--manifolds}, Progress in
Mathematics 208, Birkh\"auser Verlag, Basel, 2002.
\bibitem[Tu02b]{Tu02b}
V. Turaev, {\em A homological estimate for the Thurston norm},
unpublished note (2002), arXiv:math.~GT/0207267
\bibitem[Wa78]{Wa78}
F. Waldhausen, {\em  Algebraic K-theory of generalized free products II}, Ann.\ of
Math. 108 (1978), 205--256.
\bibitem[Wi12a]{Wi12a}
D. Wise, {\em The structure of groups with a quasi-convex hierarchy}, 189 pages, preprint (2012),
downloaded on October 29, 2012 from \\
\texttt{http://www.math.mcgill.ca/wise/papers.html}
\bibitem[Wi12b]{Wi12b}
D. Wise, {\em From riches to RAAGs: $3$-manifolds, right--angled Artin groups, and cubical geometry}, CBMS Regional Conference Series in Mathematics, 2012.


 \end{thebibliography}
\end{document}